\documentclass[11pt,reqno]{amsart}
\usepackage{amssymb,amsmath,amsthm}
\usepackage{graphicx}
\usepackage{subfigure}
\usepackage{color}
\usepackage{ulem}
\usepackage{enumitem}
\usepackage{bookmark}
\newtheorem{theorem}{Theorem}[section]
\newtheorem{lemma}{Lemma}[section]
\newtheorem{proposition}{Proposition}[section]

\usepackage{hyperref}

\theoremstyle{definition}

\makeatletter
\newcommand{\myitem}[1]{%
\item[#1]\protected@edef\@currentlabel{#1}%
}
\makeatother

\hypersetup{
    colorlinks=true,
    linkcolor=blue,
    citecolor=blue
    }

\makeatletter
\@namedef{subjclassname@2020}{2020 Mathematics Subject Classification}
\makeatother

\numberwithin{equation}{section}



\begin{document}

\title[Admissible characterization and robustness]
{Admissible characterization in delay equations and robustness of exponential dichotomies \\ against small-delay perturbations}

\maketitle

\vskip 1pt
\centerline{\scshape Shuang Chen}
\medskip
{\footnotesize
 \centerline{School of Mathematics and Statistics, Central China Normal University}
 \centerline{ Wuhan, Hubei 430079, China}
}

\vskip 0.4cm
\centerline{\scshape Weinian Zhang}
\medskip
{\footnotesize
 \centerline{School of Mathematics, Sichuan University}
 \centerline{ Chengdu, Sichuan 610064, China}
}

\medskip

\begin{abstract}

Robustness of exponential dichotomies against small-delay perturbations presents a fundamental obstacle:
the lack of an effective admissible characterization in nonautonomous delay equations.
Considerable efforts have been devoted to obtaining such a characterization for differential equations in Banach spaces.
However, even unlike ordinary differential equations,
the variation of constants formula for delay equations requires  extending the phase space to a space of discontinuous functions,
and the dependence on the past states leads to another  challenge in establishing exponential dichotomy via admissibility, i.e.,
constructing a suitable admissible pair and then estimating the growth/decay rates along the unstable/stable subspaces via the admissibility property.
In this paper, we provide an admissible characterization
based on two pairs of Banach spaces that yields explicit dichotomy exponents.
Through this characterization and an operator perturbation method,
we prove that small-delay perturbations preserve exponential dichotomies.

\end{abstract}

\parskip 0.4cm


\section{Introduction}

Exponential dichotomy provides a description of hyperbolicity for homogeneous linear differential equations,
decomposing the phase space into a direct sum of two invariant subspaces, on one of which solutions decays exponentially
and on the other of which solutions grows exponentially (see \cite{CL-99,Coppel-78,Henry-81}).
Exponential dichotomy attracts interests of research because of its applications in homoclinic and heteroclinic bifurcations,
stability analysis of nonlinear waves and linearization theory, as seen in \cite{BDP21,DZZ20,HK24,LZ22,Palmer-84,Sand-02} and references therein.

An important problem on exponential dichotomies is the robustness of the decomposition, asking
what type of perturbations preserves the decomposition.
In 1958, Massera and Sch\"affer \cite{MS-58} established the robustness of exponential dichotomies for the ODE:
\begin{eqnarray}\label{eq:ODE}
\dot x=A(t)x, \ \ \ \ x\in \mathbb{R}^{n},
\end{eqnarray}
proving that if $A(t)$ is bounded and makes (\ref{eq:ODE}) admit an exponential dichotomy, then
the perturbed equation
\[
\dot x=(A(t)+B(t))x
\]
also admits an exponential dichotomy provided that $\sup_{t\in \mathbb{R}}\|B(t)\|$ is sufficiently small.
Later, Sch\"{a}ffer \cite{Sch-63} eliminated this boundedness condition via a functional analytic argument.
A direct proof was provided by Daleckii and Krein \cite{DK-74},  but they still made the boundedness assumption.
Finally, Coppel \cite{Coppel-78} showed that this condition can be removed quite easily, offering a simpler and more elementary treatment.
These advances well established that exponential dichotomies of ODE \eqref{eq:ODE} are robust against a slight change of the coefficient matrix.

Attention is also paid to differential equations with delay, which arises naturally in many physical, biological, and engineering applications due to inevitable factors such as propagation speeds, processing times, or communication lags (\cite{Chicone03,DLP86,D88,Driver1968,Kuehn-15}).
As in \cite{DVVW,Hale-Lunel-01},
consider the linear delay equation
\begin{eqnarray}\label{NA-NDE}
\dot x= L(t)x_{t}, \ \ \ \ \ x\in\mathbb{R}^{n},
\end{eqnarray}
where $L(t): C([-r,0],\mathbb{R}^n) \to \mathbb{R}^{n}$ for each $t\in \mathbb{R}$
is a bounded
linear operator on the Banach space
$C([-r,0],\mathbb{R}^n)$, consisting of all continuous functions from $[-r,0]$ to $\mathbb{R}^{n}$ endowed with the supremum norm,
and $x_{t}$ is defined by $x_t(\theta)= x(t+\theta)$ for $\theta \in [-r,0]$.
In the case $r=0$, equation (\ref{NA-NDE}) becomes ODE \eqref{eq:ODE}, which admits an exponential dichotomy.
It is interesting to investigate whether exponential dichotomy is preserved when a small delay $r>0$ is introduced,
which is known as the robustness problem of exponential dichotomies against small-delay perturbations.
The folklore principle of delay equations, explicitly stated in Kurzweil's note \cite{Kurzweil-71} and Smith's book \cite{smith95},
says that small delays are harmless and can be ignored.
However, a growing literature now demonstrates that even arbitrarily small delays can have dramatic effects (see, e.g., \cite{BP-05,GT-14,Hale-Lunel-01}).
Up to now,
the question whether a small delay matters does not have a one-size-fits-all answer, the answer of which
depends on the different form of equations.

Nevertheless, a fundamental result, proved via complex analysis in \cite{Hale-Lunel-01} and via functional analysis in  \cite[Chapter 7]{BP-05},
indicates that exponential dichotomies of autonomous ODEs
are insensitive to autonomous small-delay perturbations,
i.e.,  ODE \eqref{eq:ODE} and equation \eqref{NA-NDE} are autonomous, with $r$ sufficiently small.
However, in the nonautonomous case, i.e., $L(t)$ is time-dependent,
the robustness problem remains unresolved. There is a considerable obstacle in tackling it,
primarily due to the lack of an effective characterization of exponential dichotomies for nonautonomous delay equations.
To address this challenge,
we provide an admissible characterization of exponential dichotomies for equation (\ref{NA-NDE}).
The notion of admissibility, originating from the pioneering work of Perron \cite{Perron-30},
was initially developed to study the stability of ODEs.
Concretely, a pair of function classes, referred to as the input class and the output class, respectively,
is said to be (properly) admissible if the nonhomogeneous ODE $\dot x = A(t)x + h(t)$ admits a (unique) solution $x$
in the output class for each $h$ in the input class (see \cite{Coppel-78, DK-74, Massera-66}).
In recent decades, the characterization of exponential dichotomies for ODEs via admissibility has witnessed significant advancements,
including the relaxation of technical hypotheses (e.g., the bounded growth condition)
and the extension to a broader spectrum of dichotomous behaviors beyond the classical one,
such as nonuniform, random, and generalized dichotomies
(see \cite{BDP25,BL-08,BDC-18,EPR25,PP10,Sasu06b,ZZ16, ZLZ17} for instance).

In the past decades,
literature has contributed to extending admissibility results from ODEs to general differential equations in Banach spaces,
as seen in \cite{CL-99,Henry-81,LevZhikov82} for an introduction and \cite{BDC-17,ChowL95,Huy-06,MH01,PPP04,PP10,Sasu06} for recent advances.
However,
relatively few works specifically cover delay equations.
Admissibility results in the context of delay equations cannot be simply obtained
by a trivial modification of the ODE arguments (see also  \cite[p.1157]{BV-20}).
Although delay equations can be viewed as infinite-dimensional dynamical systems in spaces of continuous functions,
the corresponding variation of constants formula, which is essential to establish the admissible characterization,
requires an extension of the phase space from the space of continuous functions to a space of discontinuous functions (see Section \ref{sec:ED-AD}).
This is not needed for general differential equations in Banach spaces and even ODEs.
On the other hand,
just as in ODEs (see Chapter 4 in \cite{Coppel-78}),
we need to construct suitable functions in the input class to get an admissible pair,
and then apply the admissibility property to estimate the growth and decay rates along the unstable and stable subspaces, respectively.
However, the right-hand side of \eqref{NA-NDE} depends on not only the present state but also the past states.
This therefore demands  a tailored methodology for constructing the required function pairs.

Significant progress has been made in recent years.
In 2020, Barreira and Valls \cite{BV-20} provided a complete characterization of
exponential dichotomies for nonautonomous delay equations by using the admissibility of function spaces $C_b(\mathbb{R},\mathbb{R}^n)$ and
$M(\mathbb{R},\mathbb{R}^n)$,
where $C_b(\mathbb{R},\mathbb{R}^n)$ is the set of all bounded continuous functions from $\mathbb{R}$ to $\mathbb{R}^n$,
and $M(\mathbb{R},\mathbb{R}^n)$ is the set of all measurable functions from $\mathbb{R}$ to $\mathbb{R}^n$ such that
for each $h\in M(\mathbb{R},\mathbb{R}^n)$,
$
\sup_{t\in\mathbb{R}}\int_{t}^{t+1}|h(u)|\,du<+\infty,
$
identified if two functions are equal almost everywhere.
Through this admissibility property,
Barreira and Valls \cite{BV-20b} further proved that the perturbed delay equation:
$
\dot x= \big(L(t)+N(t)\big)x_{t}
$
admits an exponential dichotomy, provided that equation \eqref{NA-NDE} with $\sup_{t\in \mathbb{R}}\int_{t}^{t+1}\|L(\tau)\|\,d\tau<+\infty$
possesses an exponential dichotomy and $\sup_{t\in \mathbb{R}}\int_{t}^{t+1}\|N(\tau)\|\,d\tau<\delta$ for sufficiently small $\delta>0$.
This extends the robustness property developed for nonautonomous ODEs in \cite{Coppel-78,DK-74,MS-58}
 to nonautonomous  delay equations.
Following the idea of \cite{BV-20,Mallet99}, in 2025
Elorreaga and G\'{o}mez \cite{Elorreaga-Gomez-25} proved that
exponential dichotomies and the Fredholm alternative
for a class of nonautonomous delay equations with asymptotically autonomous limits
can be described by the admissibility  of a suitable pair of weighted Sobolev spaces.

The admissible characterizations developed in \cite{BV-20,Elorreaga-Gomez-25} establish the existence of exponential dichotomies,
but they fail to give the explicit expressions of the dichotomy exponents.
The problem of obtaining such explicit expressions is quite different from the mere existence proof
because these expressions are indispensable for not only examining the robustness of exponential dichotomies against small-delay perturbations
but demonstrating the dependence of the dichotomy exponents on small delays.
On the other hand, the presence of delays changes the phase space from a finite-dimensional Euclidean space to the space of continuous functions,
which makes the analysis framework provided in \cite{BV-20b,Elorreaga-Gomez-25},
where the perturbations preserve the phase space, no longer applicable.

In this paper we tackle the aforementioned issues and propose a novel characterization of exponential dichotomies for nonautonomous delay equations
with the admissibility of two pairs of Banach spaces (see Section \ref{sec:admdich}).
Notably, in our proof that admissibility implies exponential dichotomy,
we obtain the dichotomy exponents by constructing two specific solutions to the nonhomogeneous equation
and subsequently extracting the exponents via the admissibility property.
The other technical difficulty arises from the presence of delays,
which we overcome by developing an operator perturbation approach to establish the proper admissibility for the perturbed delay equations.
Subsequently, we obtain the robustness of exponential dichotomies against small-delay perturbations,
and show the dependence of the dichotomy exponents on small delays.
As an easy by-product, we can choose a dichotomy exponent of the perturbed delay
equation which is continuously differentiable with respect to small delays if the corresponding kernels are uniformly bounded in time, i.e., $\sup_{t\in\mathbb{R}}\|L(t)\|<+\infty$.

The paper is organized as follows.
We summarize the main results in Section \ref{sec:mainresults}.
The problem from exponential dichotomy to the proper admissibility is proved in Section \ref{sec:ED-AD},
and the converse problem is proved in Section \ref{sec:AD-ED}.
We rigorously prove the robustness against small-delay perturbations in the final section.


\section{Setup and main results}
\label{sec:mainresults}

We start by introducing some notations used in this paper.
Let $\mathbb{R}^n$ be equipped with the norm $|\cdot|$ induced by the inner product.
Without a particular emphasis, we always use $\|\cdot\|$ to denote the norm of an element in a Banach space
or the norm of a bounded linear operator acting from one Banach space into another.

\subsection{Setup and exponential dichotomy}\

According to the Riesz representation theorem,
we can express $L(t)$ as
\[
L(t)\phi=\int^{0}_{-r} d_\theta\eta(t,\theta)\phi(\theta),
\ \ \
t\in \mathbb{R},\ \phi\in C([-r,0],\mathbb{R}^n),
\]
where $\eta=(\eta_{ij})$ is a $n\times n$ matrix valued function on $\mathbb{R}\times \mathbb{R}$.
Then delay equation \eqref{NA-NDE} becomes
\begin{eqnarray}\label{DDE}
\dot x=\int^{0}_{-r} d_\theta\eta(t,\theta)x(t+\theta), \ \ \ \ \ x\in\mathbb{R}^{n}.
\end{eqnarray}
Throughout this paper, we assume that equation \eqref{DDE} satisfies the following:
\begin{enumerate}
\item[{\bf (A1)}]  $\eta$ is continuous from the left in $\theta$ on $(-r,0)$, measurable in $(t,\theta)$,
and
\begin{eqnarray*}
\eta(t,\theta) = 0   && \mbox{ for }\  \theta\geq 0,\\
\eta(t,\theta) =\eta(t,-r) && \mbox{ for }\ \theta\leq -r;
\end{eqnarray*}
\vskip 3pt

\item[{\bf (A2)}]
$\eta$ has a bounded variation in $\theta$ on $[-r, 0]$ for each $t\in \mathbb{R}$;
\vskip 3pt

\item[{\bf (A3)}]
$L(t)\phi$ is measurable in $t\in\mathbb{R}$ for each $\phi\in C([-r,0],\mathbb{R}^n)$
and $\|L(t)\|$ is a locally integral function satisfying
\begin{eqnarray}\label{hyp-int}
\sup_{t\in \mathbb{R}} \frac{1}{r_0}\int_{t}^{t+r_0}\|L(u)\| du\leq M<+\infty.
\end{eqnarray}
\end{enumerate}
In particular, for $r_0=1$ in \eqref{hyp-int}, this setting coincides with the case studied by Barreira and Valls in \cite{BV-20}.

A continuous $x:[s-r,+\infty)\to \mathbb{R}^n$ is called a solution of delay equation \eqref{DDE}
if $x$ is absolutely continuous on  $[s-r,+\infty)$ and satisfies  \eqref{NA-NDE}
for almost every $t\in [s,+\infty)$.
By the existence and uniqueness theorem (see \cite[Theorem 1.1, p.168]{JKHale-Verduyn}),
delay equation \eqref{DDE} with initial condition
$x_s=\phi\in C([-r,0],\mathbb{R}^n)$  has a unique solution, denoted by $x_t(\cdot,s,\phi)$.
Then we can define the evolution operators $T(t,s)$ of \eqref{DDE} as
\begin{eqnarray*}
T(t,s)\phi= x_t(\cdot,s,\phi), \ \ \ t\geq s, \ \phi \in  C([-r,0],\mathbb{R}^n).
\end{eqnarray*}
We refer to \cite{DVVW,GW-13,Hale-77,JKHale-Verduyn} for more information on delay equations.

The evolution family $\{T(t,s):t\geq s\}$ on $C([-r,0],\mathbb{R}^n)$ is said to admit an {\it exponential dichotomy}
on the whole line $\mathbb{R}$ (see  \cite{CL-99,Coppel-78,JKHale-Verduyn,Hale-Zhang-04})
if there exists a family of projections $P(t)$ on $C([-r,0],\mathbb{R}^n)$ for $t\in \mathbb{R}$ such that
\begin{enumerate}
\item[{\bf (E1)}] $T(t,s)P(s)=P(t)T(t,s)$ for $t\geq s$;
\vskip 3pt

\item[{\bf (E2)}] $T(t,s)|_{\mathcal{N}(P(s))}$ is an isomorphism from the kernel $\mathcal{N}(P(s))$ of $P(s)$ onto $\mathcal{N}(P(t))$ for $t\geq s$,
which defines the inverse of $T(t,s)|_{\mathcal{N}(P(s))}$, denoted by $T(s,t):\mathcal{N}(P(t)) \to \mathcal{N}(P(s))$;
\vskip 3pt

\item[{\bf (E3)}] There exist constants $K>0$ and $\alpha>0$ such that
\begin{eqnarray*}
\|T(t,s)P(s)\|\leq Ke^{-\alpha(t-s)},&& \ \ \ t\geq s,\\
\|T(t,s)(Id-P(s))\|\leq Ke^{-\alpha(s-t)}, && \ \ \ s\geq t,
\end{eqnarray*}
where $Id$ is the identity operator.
\end{enumerate}
We call $\alpha$ and $K$ an {\it exponent} and a {\it bound} of this dichotomy respectively,
and $P(t)$  the {\it dichotomy projection}.
For notational convenience,
we denote this dichotomy by $\mathcal{E}(\alpha,K)$.
The exponential dichotomy for ODE \eqref{eq:ODE} on the whole line $\mathbb{R}$ is defined in the same way,
and without confusion, we will use the same notations.


\subsection{Characterization via admissibility}\
\label{sec:admdich}

We now provide an admissible characterization of exponential dichotomies for nonautonomous delay equations.
Given $b\in \mathbb{R}$, let $\mathcal{H}_{b}(\mathbb{R})$
denote the set of locally integrable functions $h: \mathbb{R}\to \mathbb{R}^{n}$
such that
\[
\sup_{t\in\mathbb{R}} e^{b |t|}\int_{t}^{t+r_0}|h(\tau)|\,d\tau<+\infty,
\]
and $\mathcal{Y}_{b}(\mathbb{R})$
the set of continuous functions $y: \mathbb{R}\to \mathbb{R}^{n}$
such that
\[
\sup_{t\in\mathbb{R}} e^{b |t|}|y(t)|<+\infty.
\]
One can check that $\mathcal{H}_{b}(\mathbb{R})$ and $\mathcal{Y}_{b}(\mathbb{R})$ are Banach spaces equipped with the norms
\begin{eqnarray*}
\begin{aligned}
\|h\|_{\mathcal{H}_{b}(\mathbb{R})} := & \sup_{t\in\mathbb{R}} \frac{1}{r_0}e^{b |t|}\int_{t}^{t+r_0}|h(\tau)|\,d\tau,~~~ h\in \mathcal{H}_{b}(\mathbb{R}),
\\
\|y\|_{\mathcal{Y}_{b}(\mathbb{R})} := & \sup_{t\in\mathbb{R}} e^{b |t|}|y(t)|,~~~ y\in \mathcal{Y}_{b}(\mathbb{R}),
\end{aligned}
\end{eqnarray*}
respectively.

Consider the nonhomogeneous delay equation:
\begin{eqnarray}\label{DDEnonhomo}
\dot y= L(t)y_t+h(t),
\end{eqnarray}
where $L(t)$ are defined before \eqref{DDE} for $t\in \mathbb{R}$.
For any given $h \in \mathcal{H}_{b}(\mathbb{R})$,
by Theorem 1.1 of \cite[page 168]{JKHale-Verduyn},
equation \eqref{DDEnonhomo} associated with the initial condition $x_{s}=\phi\in C([-r,0],\mathbb{R}^n)$
has a unique solution.
As shown in \cite{BDC-18, BV-20, Coppel-78, DK-74, Massera-66},
the pair $(\mathcal{H}_{b}(\mathbb{R}),\mathcal{Y}_{b}(\mathbb{R}))$ for some $b\in \mathbb{R}$ is said to be {\it  (properly) admissible}
with respect to
equation \eqref{DDEnonhomo}
if equation \eqref{DDEnonhomo} has a (unique) solution $x\in \mathcal{Y}_{b}(\mathbb{R})$ for each $h \in \mathcal{H}_{b}(\mathbb{R})$.
The following result shows that we can characterize exponential dichotomies of delay equation \eqref{DDE} in terms of an admissibility property.

\begin{theorem}\label{thm:ed-admis}
Equation \eqref{DDE} admits an exponential dichotomy on the whole line $\mathbb{R}$ if and only if
there exists a pair of constants $b^-$ and $b^+$ with $b^-<0<b^+$
such that both $(\mathcal{H}_{b^+}(\mathbb{R}),\mathcal{Y}_{b^+}(\mathbb{R}))$ and $(\mathcal{H}_{b^-}(\mathbb{R}),\mathcal{Y}_{b^-}(\mathbb{R}))$
are properly admissible
with respect to equation \eqref{DDEnonhomo}.
\end{theorem}

This theorem is proved by using Proposition \ref{prop-ED-to-Admis} in Section \ref{sec:ED-AD}
and  Proposition \ref{prop-adm-to-NED} in Section \ref{sec:AD-ED},
one shows that exponential dichotomies imply the proper admissibility of $(\mathcal{H}_{b^{\pm}}(\mathbb{R}),\mathcal{Y}_{b^{\pm}}(\mathbb{R}))$,
and the other shows the converse result.
It gives the admissible characterization of exponential dichotomies for delay equation \eqref{DDE} on the whole line $\mathbb{R}$.
The same problem on the half-line can be treated by an analogous discussion, requiring only a slight modification.

We emphasize that the admissible characterization obtained in Theorem  \ref{thm:ed-admis} is applicable to ODEs as a special case,
for which arguments analogous to those of (\cite{Coppel-78,DZZ20,Huy-06,ZLZ17}) are available.
However, the extension to delay equations is by no means trivial.
Indeed, just as in the proofs of Lemma 4.4 in \cite{BV-20} and Theorem 1.2 in \cite{Elorreaga-Gomez-25},
a key step in establishing the admissible characterization of delay equations is
to construct a suitable function in the input class so as to obtain an admissible pair and then establish the dichotomy decomposition of the phase space.
The methods used in \cite{Coppel-78,DZZ20,Huy-06,ZLZ17} are inapplicable to this procedure.
To tackle this issue and the robustness problem,
we construct two novel functions in $\mathcal{H}_{b^\pm}(\mathbb{R})$ for $b^-<0<b^+$ (see \eqref{df:hstar} and \eqref{df:hstar2}).
Using these functions,
we obtain two admissible pairs which allow us not only to establish the admissibility characterization
but also to derive explicit expressions for the dichotomy exponents.


\subsection{Robustness of exponential dichotomies}\

Another goal of this paper is to study the robustness of exponential dichotomies against small-delay perturbations.
We now make an additional assumption about the kernel $\eta$:
\begin{enumerate}
\item[{\bf (A4)}] $\eta(t,\theta) =-A(t)$  for all $\theta\leq -r$ and $t\in \mathbb{R}$,
and the delay $r$ satisfies $0<r<r_{0}$.
\end{enumerate}
The constant $r_0$ in {\bf (A4)} will be used to indicate the {\it smallness condition} of delay $r$.
This assumption ensures that we can regard equation \eqref{DDE}  as a small-delay perturbation of ODE \eqref{eq:ODE}.
By {\bf (A3)} and {\bf (A4)}, we have
\begin{eqnarray}\label{hyp:A}
\sup_{t\in\mathbb{R}}\frac{1}{r_0}\int_{t}^{t+r_0}\|A(\tau)\|\,d\tau \leq M<+\infty.
\end{eqnarray}
A simple example meeting {\bf (A4)} is the delay equation
\[
\dot x = A(t)x(t-r),
\]
which is of the form \eqref{DDE} with $L(t)\phi=A(t)\phi(-r)$ for each $\phi\in C([-r,0],\mathbb{R}^n)$
and the kernel
\begin{eqnarray*}
\eta(t,\theta)=
\left\{
\begin{aligned}
& -A(t),\ \ \ && \mbox{ if }\ \theta\leq -r,\\
& 0,\ \ \   &&\mbox{ if }\ \theta >-r.
\end{aligned}
\right.
\end{eqnarray*}

In particular,
if $\|L(t)\|$ is uniformly bounded with respect to $t\in \mathbb{R}$, i.e.,
there exists $M\in (0,\infty)$ such that $\sup_{t\in\mathbb{R}}\|L(t)\|\leq M<+\infty$.
Then we see that condition \eqref{hyp-int} holds for all $r_0>0$ because
\begin{eqnarray*}
\frac{1}{r_0}\int_{t}^{t+r_0}\|L(u)\| du
  \leq \frac{1}{r_0}\int_{t}^{t+r_0}M\,du
  \leq M<+\infty,
  \ \ \ \forall\, t\in\mathbb{R}.
\end{eqnarray*}
Hence, equation \eqref{DDE}, subject to the aforementioned settings,
also includes  small-delay problems considered in
 \cite{Chen-Shen-2020,Driver1968,Driver1976}
and asymptotically autonomous delay equations in \cite{Elorreaga-Gomez-25,Lin-86}.

Finally, we are in the position to state two theorems on the robustness of exponential dichotomies against small-delay perturbations.
The detailed proofs are given in Section \ref{sec:pfs}.

\begin{theorem}\label{thm-ED}
Suppose that ODE \eqref{eq:ODE} admits an exponential dichotomy
$\mathcal{E}(\alpha,K)$ on the whole line $\mathbb{R}$ and delay equation \eqref{DDE} satisfies {\bf (A1)}\,-\,{\bf (A4)}.
Then there is $r_0$ satisfying $\ell(\alpha,r_0) r_0<\alpha$,
where $\ell(\alpha,r_0):=Me^{4\alpha r_0}(4KM+\alpha)$,
such that for any $\beta \in (0, \alpha-\ell(\alpha,r_0) r_0]$ equation \eqref{DDE} with delay $r\in (0,r_0]$
admits an exponential dichotomy on the whole line $\mathbb{R}$ with exponent $\beta$.
\end{theorem}

Inequality $\ell(\alpha,r_0) r_0<\alpha$ indicates the smallness condition, i.e., requiring that $r$ is small enough such that $0<r\leq r_0$.
This condition is not necessarily optimal and the sharp condition will be investigated in a future work.

If  $\|L(t)\|$ is uniformly bounded with respect to $t$, we further have the following result.

\begin{theorem}\label{thm-ED2}
Suppose that  ODE \eqref{eq:ODE} admits an exponential dichotomy $\mathcal{E}(\alpha,K)$ on the whole line $\mathbb{R}$,
and delay equation \eqref{DDE} satisfies {\bf (A1)}\,-\,{\bf (A4)} with the boundedness $\sup_{t\in\mathbb{R}}\|L(t)\|\leq M$ for some $M>0$.
Then for any $r\in (0,\varepsilon_0)$ and $\beta \in (0, \alpha-\zeta(\alpha) r]$,
where
$\varepsilon_0 :=\min\{1,\, \alpha /\zeta(\alpha)\}>0$
and $\zeta(\alpha) :=Me^{4\alpha }(4KM+\alpha)$,
equation \eqref{DDE} with delay $r$ admits an exponential dichotomy on the whole line $\mathbb{R}$ with exponent $\beta$.
\end{theorem}

This theorem demonstrates that
we can choose  a dichotomy exponent of equation \eqref{DDE} in the affine form
\[
\beta(r)=\alpha-\zeta(\alpha)r,\qquad r\in(0,\varepsilon_0),
\]
which depends linearly and $C^1$ on the small delay $r$, provided that $\sup_{t\in\mathbb{R}}\|L(t)\|\leq M$.
This property, to our best knowledge,
has not been reported for small-delay perturbations of nonautonomous differential equations in existing references.
Finally, we remark that these two theorems provide a characterization of exponential dichotomies for delay equation \eqref{DDE}
in terms of ODE \eqref{eq:ODE}. This allows us to further characterize the dichotomy spectrum (see \cite{SS-78,Siegmund-02}) via  ODE \eqref{eq:ODE},
which is known to be important to establish linearization theorems for nonautonomous differential equations
(see, for instance, \cite{BV-20b,CDS-19,DZZ20, Palmer-73}).


\section{Exponential dichotomy implies proper admissibility}\
\label{sec:ED-AD}

In this section, we prove the necessity of Theorem \ref{thm:ed-admis}.
We start with reviewing the variation of constants formula for delay equations
and also refer to \cite{BV-20b,Hale-77,JKHale-Verduyn} for more information.

Let $L_{1}^{\rm loc}(\mathbb{R},\mathbb{R}^n)$ denote the space of functions from $\mathbb{R}$ into $\mathbb{R}^n$
which are Lebesgue integrable on every compact set of $\mathbb{R}$.
Note from Theorem 2.1 of \cite{Hale-77} that
for any $h\in L_{1}^{\rm loc}(\mathbb{R},\mathbb{R}^n)$
the solution representation of nonhomogeneous delay equation \eqref{DDEnonhomo} with initial condition \(y_t=\phi\in C([-r,0],\mathbb R^n)\),
which is obtained by the Riesz representation theorem,
involves a discontinuous matrix-valued function.
Consequently, unlike ODEs, the variation of constants formula for delay equation \eqref{DDE}
cannot be formulated on the original phase space directly but, instead,
it requires to extend the evolution operator $T(t,s)$ of \eqref{DDE} to a larger space of discontinuous functions.

Proceeding as in \cite{Hale-77}, we extend $C([-r,0],\mathbb{R}^n)$ to the set $C_{0}([-r,0],\mathbb{R}^n)$
of all functions $\varphi$ from $[-r,0]$ to $\mathbb{R}^n$ such that
each $\varphi \in C_{0}([-r,0],\mathbb{R}^n)$ is continuous on $[-r,0)$ and the limit $\lim_{\theta\to 0^-}\varphi(\theta)$ exists.
One can verify that $C_{0}([-r,0],\mathbb{R}^n)$ is a Banach space equipped with the supremum norm.
Equation \eqref{DDE} with initial condition
$x_{s}=\varphi\in C_{0}([-r,0],\mathbb{R}^n)$ has a unique solution, denoted by $x_t(\cdot,s,\varphi)$.
Accordingly, we define the evolution operator $T_{0}(t,s)$ as
\begin{eqnarray*}
T_{0}(t,s)\varphi= x_t(\cdot,s,\varphi), \ \ \ t\geq s, \ \varphi \in  C_{0}([-r,0],\mathbb{R}^n).
\end{eqnarray*}
Then $\{T_{0}(t,s):t\geq s\}$  forms a strongly continuous semigroup on $C_{0}([-r,0],\mathbb{R}^n)$.
In particular, $T_0(t,s)\phi=T(t,s)\phi$ for each $\phi \in C([-r,0],\mathbb{R}^n)$ and $t\geq s$ (see \cite{BV-20,Hale-77}).

For any $s\in \mathbb{R}$ and $\phi\in C([-r,0],\mathbb{R}^{n})$,
equation \eqref{DDEnonhomo} with  $y_{s}=\phi$ has a unique solution $y$ given by
\begin{eqnarray}\label{eq:CVF}
y_t=T(t,s)\phi+\int_{s}^{t}T_0(t,u)X_0 h(u)\,du, \ \ \ t\geq s,
\end{eqnarray}
where
$X_0: \mathbb{R}^n \to C_{0}([-r,0],\mathbb{R}^{n})$ is defined by
\begin{eqnarray*}
(X_0\xi)(\theta):=
\left\{
\begin{aligned}
& 0,\ \ \ && \mbox{ if }\ \theta\in [-r,0),\\
& \xi,\ \ \   &&\mbox{ if }\ \theta =0,
\end{aligned}
\right.
\qquad \xi\in \mathbb{R}^{n}.
\end{eqnarray*}
This is called the {\it variation of constants formula} for equation \eqref{DDEnonhomo} (see \cite[Section 6.2]{Hale-77}).
Note that $X_0$ has a jump discontinuity at $\theta=0$.
Therefore, $X_0\xi$ belongs to $C_{0}([-r,0],\mathbb{R}^{n})$ but not to $C([-r,0],\mathbb{R}^{n})$.
With the extension, the integral term has a rigorous meaning in the discontinuous space $C_0([-r,0],\mathbb{R}^n)$.

The extension is also indispensable for the projected variation-of-constants formulas used in the study of exponential dichotomies.
Actually,
we can split formula \eqref{eq:CVF} into two parts, provided that equation \eqref{DDE} admits an exponential dichotomy
with the dichotomy projection $P(t)$ for each $t\in\mathbb{R}$.
Precisely, let $P_0(t), Q_0(t):\mathbb{R}^n\to C([-r,0],\mathbb{R}^{n})$ denote two maps of the form
\[
P_0(t)=X_0-Q_0(t),
\ \ \
Q_0(t)=T(t,t+r)Q(t+r)T_0(t+r,t)X_0,
\]
where $Q(t)=Id-P(t)$.
Then the projections of the integral term in \eqref{eq:CVF} onto the stable and unstable subspaces satisfy (see (49) and (50) in \cite{BV-20b})
\begin{eqnarray}\label{eq:prj-stab}
\begin{aligned}
P(t)\int_{\tau}^{t}T_0(t,u)X_0 h(u)\,du =& \int_{\tau}^{t}T_0(t,u)P_0(u) h(u)\,du, \\
Q(t)\int_{\tau}^{t}T_0(t,u)X_0 h(u)\,du =& \int_{\tau}^{t}T(t,u)Q_0(u) h(u)\,du,
\end{aligned}
\end{eqnarray}
respectively. Accordingly, $x_t$ in \eqref{eq:CVF} satisfies
\begin{eqnarray}\label{eq:soluPrj}
\begin{aligned}
P(t)x_t&= T(t,s)P(s)\phi+\int_{s}^{t}T_0(t,u)P_0(u) h(u)\,du, \\
Q(t)x_t&= T(t,s)Q(s)\phi+\int_{s}^{t}T(t,u)Q_0(u)h(u)\,du
\end{aligned}
\end{eqnarray}
for all $t\geq s$ in $\mathbb{R}$. In order to estimate the above integrals, we require the following lemma.

\begin{lemma}{\rm (see \cite[Proposition 1]{BV-20b})}\label{lm:est-prj}
Suppose that equation \eqref{DDE} admits an exponential dichotomy $\mathcal{E}(\alpha,K)$ on $\mathbb{R}$.
Then there exists a constant $N\geq 1$ such that
\begin{eqnarray*}
&\|T_0(t,s)P_0(s)\|\leq Ne^{-\alpha(t-s)},& \ \ \ t \geq s,\\
&\|T(t,s)Q_0(s)\|  \leq   Ne^{-\alpha(s-t)}, & \ \ \ s \geq t.
\end{eqnarray*}
\end{lemma}

The following result shows that exponential dichotomy implies proper admissibility.

\begin{proposition}\label{prop-ED-to-Admis}
Suppose that equation \eqref{DDE} admits an exponential dichotomy $\mathcal{E}(\alpha,K)$ on $\mathbb{R}$.
Then for any $b\in (-\alpha,\alpha)$,
the pair  $(\mathcal{H}_{b}(\mathbb{R}),\mathcal{Y}_{b}(\mathbb{R}))$ is properly admissible
with respect to equation \eqref{DDEnonhomo}.
\end{proposition}
\begin{proof}
Suppose that equation \eqref{DDE} admits an exponential dichotomy $\mathcal{E}(\alpha,K)$ on $\mathbb{R}$
with the dichotomy projection $P(t)$ for each $t\in\mathbb{R}$.
Fix any $b\in (-\alpha,\alpha)$ and take each $h\in \mathcal{H}_{b}(\mathbb{R})$.
Set
\[
y_{t}:=(\mathcal{I}_{1}h)_{t}+(\mathcal{I}_{2}h)_{t}, \ \ \ \ \ t\in \mathbb{R},
\]
where
\begin{eqnarray*}
\begin{aligned}
(\mathcal{I}_{1}h)_{t}&=\int^{t}_{-\infty}T_{0}(t,u)P_{0}(u)h(u)\,du,\\
(\mathcal{I}_{2}h)_{t}&= -\int^{+\infty}_{t}T(t,u)Q_{0}(u)h(u)\,du.
\end{aligned}
\end{eqnarray*}
Since $P(t)$ is continuous in $t$ (see, for example,  \cite[Lemma 3.1]{DZZ-22}),
we have that  $y$, $\mathcal{I}_{1}h$ and $\mathcal{I}_{2}h$ are continuous functions on $\mathbb{R}$ and linear in $h$.

By Lemma \ref{lm:est-prj}, for any $t\leq 0$ we have
\begin{eqnarray}\label{est:stabinteg-1}
\begin{aligned}
\|(\mathcal{I}_{1}h)_{t}\|
&\leq  \int_{-\infty}^{t} Ne^{-\alpha(t-u)}|h(u)|\,du \\
&\leq N \sum_{m=0}^{\infty}\int^{t-m r_0}_{t-(m+1)r_0} e^{-\alpha m r_0}|h(u)|\,du\\
& = N r_0  e^{-b r_0} e^{b t}\sum_{m=0}^{\infty}e^{-(\alpha+b) mr_0} \left(\frac{1}{r_0}e^{b|t-(m+1)r_0|}\int^{t-m r_0}_{t-(m+1)r_0} |h(u)|\,du \right)\\
&\leq N r_0  e^{-b r_0} e^{b t} \sum_{m=0}^{\infty}e^{-(\alpha+b)m r_0}\|h\|_{\mathcal{H}_{b}(\mathbb{R})}\\
&\leq \frac{N r_0  e^{-b r_0}}{1-e^{-(\alpha+b) r_0}}\|h\|_{\mathcal{H}_{b}(\mathbb{R})}e^{b t},
\end{aligned}
\end{eqnarray}
where we use $\alpha+b>0$ and $h\in \mathcal{H}_{b}(\mathbb{R})$.
Let  $[\,\cdot\,]$ denote the integer part of a real number.
For any $t>0$, using \eqref{est:stabinteg-1} and $t-r_0<[\frac{t}{r_0}] r_0\leq t$, we have
\[
\begin{aligned}
\|(\mathcal{I}_{1}h)_{t}\|
&\leq  \|\int^{0}_{-\infty}T_{0}(t,u)P_{0}(u)h(u)\,du\|+\|\int^{t}_{0}T_{0}(t,u)P_{0}(u)h(u)\,du\| \\
&\leq  \int_{-\infty}^{0} N e^{-\alpha(t-u)}|h(u)|\,du +\int_{0}^{t} N e^{-\alpha(t-u)}|h(u)|\,d u \\
&\leq  \frac{N r_0  e^{-b r_0}}{1-e^{-(\alpha+b) r_0}}\|h\|_{\mathcal{H}_{b}(\mathbb{R})} e^{-\alpha t}
  + N\sum_{m=0}^{[\frac{t}{r_0}]} \int_{([\frac{t}{r_0}]-m)r_0}^{([\frac{t}{r_0}]+1-m)r_0} e^{-\alpha ([\frac{t}{r_0}]r_0-u)}|h(u)|\,du \\
&\leq  \frac{N r_0  e^{-b r_0}}{1-e^{-(\alpha+b) r_0}}\|h\|_{\mathcal{H}_{b}(\mathbb{R})} e^{-\alpha t}
  +N\sum_{m=0}^{[\frac{t}{r_0}]} e^{-\alpha(m-1)r_0} \int_{([\frac{t}{r_0}]-m)r_0}^{([\frac{t}{r_0}]+1-m)r_0}|h(u)|\,d u.
\end{aligned}
\]
For the second term in the last inequality, noting that $h\in\mathcal{H}_{b}(\mathbb{R})$ and $b-\alpha<0$, we get
\[
\begin{aligned}
& N\sum_{m=0}^{[\frac{t}{r_0}]} e^{-\alpha(m-1)r_0} \int_{([\frac{t}{r_0}]-m)r_0}^{([\frac{t}{r_0}]+1-m)r_0}|h(u)|\,du \\
& \leq   N r_0 e^{(\alpha+|b|) r_0} \sum_{m=0}^{[\frac{t}{r_0}]} e^{-(\alpha-b)mr_0-b t} \left( \frac{1}{r_0}e^{b([\frac{t}{r_0}]-m)r_0} \int_{([\frac{t}{r_0}]-m)r_0}^{([\frac{t}{r_0}]+1-m)r_0}|h(u)|\,d u \right)\\
&\leq  N r_0 e^{(\alpha+|b|) r_0} \sum_{m=0}^{[\frac{t}{r_0}]} e^{-(\alpha-b)mr_0-b t} \|h\|_{\mathcal{H}_{b}(\mathbb{R})} \\
&\leq  \frac{N r_0 e^{(\alpha+|b|) r_0} }{1-e^{-(\alpha-b)r_0}}\|h\|_{\mathcal{H}_{b}(\mathbb{R})}e^{-bt},
\ \ \ \ t>0.
\end{aligned}
\]
Accordingly, we have
\[
\|(\mathcal{I}_{1}h)_{t}\|
\leq \frac{2N r_0 e^{(\alpha+|b|) r_0} }{1-e^{-(\alpha-|b|)r_0}}\|h\|_{\mathcal{H}_{b}(\mathbb{R})}e^{-bt}
\ \ \ \mbox{ for }  t>0.
\]
This together with \eqref{est:stabinteg-1} yields that for each $h\in\mathcal{H}_{b}(\mathbb{R})$,
\begin{eqnarray}\label{est:stabinteg-2}
\begin{aligned}
\|\mathcal{I}_{1}h\|_{\mathcal{Y}_{b}(\mathbb{R})}
\leq \frac{2 N  r_0 e^{(\alpha+|b|) r_0} }{1-e^{-(\alpha-|b|)r_0}}\|h\|_{\mathcal{H}_{b}(\mathbb{R})},
\end{aligned}
\end{eqnarray}
where we use the fact that
\[
\|(\mathcal{I}_{1}h)_{t}\|=\sup_{\theta\in [-r,0]}|(\mathcal{I}_{1}h)(t+\theta)| \geq |(\mathcal{I}_{1}h)(t)|.
\]

For  $(\mathcal{I}_{2}h)_{t}$, we have that for $t\geq 0$,
\begin{eqnarray}\label{est:unstabinteg-1}
\begin{aligned}
\|(\mathcal{I}_{2}h)_{t}\|
& \leq  \int^{\infty}_{t} N e^{-\alpha(u-t)}|h(u)|\,du \\
& \leq N \sum_{m=0}^{\infty}\int_{t+m r_0}^{t+(m+1)r_0} e^{-\alpha m r_0}|h(u)|\,du \\
& = N r_0 e^{-b t}\sum_{m=0}^{\infty} e^{-(\alpha+b) m r_0} \left( \frac{1}{r_0}e^{b(t+m r_0)}\int^{t+(m+1)r_0}_{t+m r_0}|h(u)|\,du \right)\\
& \leq \frac{N r_0}{1-e^{-(\alpha+b)r_0}}\|h\|_{\mathcal{H}_{b}(\mathbb{R})}e^{-b t},
\end{aligned}
\end{eqnarray}
where we use $\alpha+b>0$ and $h\in \mathcal{H}_{b}(\mathbb{R})$.
For $t<0$, using \eqref{est:unstabinteg-1}, we have
\[
\begin{aligned}
\|(\mathcal{I}_{2}h)_{t}\|
&\leq  \|\int^{+\infty}_{0}T(t,u)Q_{0}(u)h(u)\,du\|+\|\int^{0}_{t}T(t,u)Q_{0}(u)h(u)\,du\|\\
&\leq  \int^{\infty}_{0} Ne^{-\alpha(u-t)}|h(u)|\,d u+\int^{0}_{t} Ne^{-\alpha(u-t)}|h(u)|\,d u\\
&\leq \frac{N r_0}{1-e^{-(\alpha+b)r_0}}\|h\|_{\mathcal{H}_{b}(\mathbb{R})}e^{\alpha t}
  +N\sum_{m=0}^{[-\frac{t}{r_0}]}\int_{(m-1-[-\frac{t}{r_0}])r_0}^{(m-[-\frac{t}{r_0}])r_0}e^{-\alpha(u+[-\frac{t}{r_0}]r_0)}|h(u)|\,d u,
\end{aligned}
\]
where we use the fact that $t\leq -[-\frac{t}{r_0}]r_0<t+r_0$ in the last line.
As for the second term in the last inequality, we have
\[
\begin{aligned}
& N\sum_{m=0}^{[-\frac{t}{r_0}]}\int_{(m-1-[-\frac{t}{r_0}])r_0}^{(m-[-\frac{t}{r_0}])r_0}e^{-\alpha(u+[-\frac{t}{r_0}]r_0)}|h(u)|\,d u\\
&\ \ \ \leq N \sum_{m=0}^{[-\frac{t}{r_0}]}e^{-\alpha(m-1)r_0}\int_{(m-1-[-\frac{t}{r_0}])r_0}^{(m-[-\frac{t}{r_0}])r_0}|h(u)|\,du \\
&\ \ \ = N r_0 e^{(\alpha-b) r_0-b [-\frac{t}{r_0}]r_0}\sum_{m=0}^{[-\frac{t}{r_0}]}e^{-(\alpha-b) m r_0}\left( \frac{1}{r_0}e^{b([-\frac{t}{r_0}]-m+1)r_0}\int_{(m-1-[-\frac{t}{r_0}])r_0}^{(m-[-\frac{t}{r_0}])r_0}|h(u)|\,du \right) \\
&\ \ \ \leq N r_0 e^{(\alpha+|b|) r_0+b t}\sum_{m=0}^{[-\frac{t}{r_0}]}e^{-(\alpha-b) m r_0}\left( \frac{1}{r_0}e^{b([-\frac{t}{r_0}]-m+1)r_0}\int_{(m-1-[-\frac{t}{r_0}])r_0}^{(m-[-\frac{t}{r_0}])r_0}|h(u)|\,du \right) \\
&\ \ \ \leq \frac{N r_0 e^{(\alpha+|b|) r_0}}{1-e^{-(\alpha-b)r_0}}\|h\|_{\mathcal{H}_{b}(\mathbb{R})}e^{b t}, \ \ \ \ t< 0.
\end{aligned}
\]
Accordingly, we have
\[
\begin{aligned}
\|(\mathcal{I}_{2}h)_{t}\|
\leq \frac{2 N r_0 e^{(\alpha+|b|) r_0}}{1-e^{-(\alpha-|b|)r_0}}\|h\|_{\mathcal{H}_{b}(\mathbb{R})}e^{b t}, \ \ \ \ t< 0.
\end{aligned}
\]
This together with \eqref{est:unstabinteg-1} yields that for each $h\in\mathcal{H}_{b}(\mathbb{R})$,
\begin{eqnarray}\label{est:unstabinteg-2}
\begin{aligned}
\|\mathcal{I}_{2}h\|_{\mathcal{Y}_{b}(\mathbb{R})}
\leq  \frac{2 N r_0 e^{(\alpha+|b|) r_0}}{1-e^{-(\alpha-|b|)r_0}}\|h\|_{\mathcal{H}_{b}(\mathbb{R})}.
\end{aligned}
\end{eqnarray}
By \eqref{est:stabinteg-2} and \eqref{est:unstabinteg-2},
we see that
$y:\mathbb{R}\to \mathbb{R}^{n}$ is well-defined on $\mathbb{R}$ and $y\in \mathcal{Y}_{b}(\mathbb{R})$.
Furthermore, we can verify that $y$ is continuous and satisfies equation \eqref{DDEnonhomo}.

To prove the uniqueness of the solution $y$,
it suffices to prove that if  $x_{t}=T(t,s)x_{s}$ for $t\geq s$ such that $x\in \mathcal{Y}_{b}(\mathbb{R})$, then $x=0$. Fix any $t\in \mathbb{R}$ and let
\[
V_{t}=P(t)x_{t}, \ \ \ U_{t}=Q(t)x_{t}.
\]
By \eqref{eq:soluPrj}, we have that  for any $s\geq 0$,
\[
V_{t}=P(t)T(t,t-s)x_{t-s}=T(t,t-s)P(t-s)V_{t-s}.
\]
Since equation \eqref{DDE} admits $\mathcal{E}(\alpha,K)$ on $\mathbb{R}$
and $V\in \mathcal{Y}_{b}(\mathbb{R})$,
we have the following estimates:
\[
\|V_{t}\|
\leq Ke^{-\alpha s} \|V\|_{\mathcal{Y}_{b}(\mathbb{R})}e^{-b|t-s|}
\leq K\|V\|_{\mathcal{Y}_{b}(\mathbb{R})}e^{bt} e^{-(\alpha+b)s}
\]
for all $s\geq \max\{0,t\}$. Letting $s\to +\infty$ yields $V_{t}=0$
because $\alpha+b>0$ for any $b\in (-\alpha,\alpha)$.
Similarly,  we have $U_t=0$.
Then we can conclude that $x=0$. This completes the proof.
\end{proof}


\section{Proper admissibility implies exponential dichotomy}\
\label{sec:AD-ED}

Let $b^+$ and $b^-$ be two real constants such that
$
b^-<0<b^+.
$
We can verify the following inclusions:
\begin{eqnarray*}
\mathcal{H}_{b^+}(\mathbb{R})\subset \mathcal{H}_{b^-}(\mathbb{R}),
\ \ \
\mathcal{Y}_{b^+}(\mathbb{R})\subset \mathcal{Y}_{b^-}(\mathbb{R}).
\end{eqnarray*}
In this section, we prove that the proper admissibility of two pairs $(\mathcal{H}_{b^\pm}(\mathbb{R}),\mathcal{Y}_{b^\pm}(\mathbb{R}))$ for equation \eqref{DDEnonhomo}
implies an exponential dichotomy of equation \eqref{DDE},
i.e., the sufficiency of Theorem \ref{thm:ed-admis}.
More precisely, we have the following result.

\begin{proposition}\label{prop-adm-to-NED}
Suppose that there exists a pair of constants $b^-$ and $b^+$ satisfying
\[
b^-<0<b^+
\]
such that both $(\mathcal{H}_{b^+}(\mathbb{R}),\mathcal{Y}_{b^+}(\mathbb{R}))$ and $(\mathcal{H}_{b^-}(\mathbb{R}),\mathcal{Y}_{b^-}(\mathbb{R}))$ are properly admissible
with respect to equation \eqref{DDEnonhomo}. Then equation \eqref{DDE} admits an exponential dichotomy on $\mathbb{R}$
with exponent $\beta=\min\{b^+,\,-b^-\}$.
\end{proposition}

The proof of this proposition needs the following four preparatory steps.

{\bf Step 1.} {\it Introduce two auxiliary operators by the proper admissibility.}

For each $h\in\mathcal{H}_{b^\pm}(\mathbb{R})$,
the proper admissibility of $(\mathcal{H}_{b^\pm}(\mathbb{R}),\mathcal{Y}_{b^\pm}(\mathbb{R}))$ implies that
equation \eqref{DDEnonhomo} has a unique solution $y\in \mathcal{Y}_{b^\pm}(\mathbb{R})$.
Then we define $\mathcal{A}^\pm: \mathcal{H}_{b^\pm}(\mathbb{R}) \to \mathcal{Y}_{b^\pm}(\mathbb{R})$ as
\begin{eqnarray}
\label{AAA}
\mathcal{A}^{\pm} h:=y.
\end{eqnarray}

\begin{lemma}\label{lm:bddoperator}
$\mathcal{A}^{\pm}:\mathcal{H}_{b^\pm}(\mathbb{R})\to \mathcal{Y}_{b^\pm}(\mathbb{R})$ are bounded linear operators.
\end{lemma}
\begin{proof}
We provide only the proof for $\mathcal{A}^-$,
as the proof for $\mathcal{A}^+$ is analogous.
It follows from  the superposition principle that  $\mathcal{A}^-$ is linear.
To prove that $\mathcal{A}^-$ is bounded,
it suffices to prove that $\mathcal{A}^-$ has a closed graph.
For a sequence $\{h_{n}\}_{n\in \mathbb{Z}^{+}}$ of $\mathcal{H}_{b^-}(\mathbb{R})$,
suppose that $h_{n}\to h$ in $\mathcal{H}_{b^-}(\mathbb{R})$ and $y_{n}=\mathcal{A}h_{n}\to y$ in $\mathcal{Y}_{b^-}(\mathbb{R})$ as $n\to +\infty$.
Then
\[
y_{n}(t)-y_{n}(s)= \int_{s}^{t} \int^{0}_{-r} d_{\theta}\eta(\tau,\theta)y_{n}(\tau+\theta)\,d\tau + \int_{s}^{t} h_{n}(\tau)d\tau
\ \ \ \mbox{ for $t\geq s$}.
\]
Noting that $h_n\in \mathcal{H}_{b^-}(\mathbb{R})$ and $y_n\in \mathcal{Y}_{b^-}(\mathbb{R})$,
we can compute that
\begin{eqnarray}\label{ineq:est-h-y}
\begin{aligned}
|y_{n}(t)-y(t)|&\leq e^{-b^-|t|}\|y_{n}-y\|_{\mathcal{Y}_{b^-}(\mathbb{R})},
\\
\int_{t}^{t+r_0}|h_{n}(\tau)-h(\tau)|\,d\tau &\leq r_0 e^{-b^-|t|}\|h_{n}-h\|_{\mathcal{H}_{b^-}(\mathbb{R})}.
\end{aligned}
\end{eqnarray}
Then for any fixed $t\geq s$ in $\mathbb{R}$,
\begin{eqnarray*}
&& \left| \int_{s}^{t} \int^{0}_{-r} d_{\theta}\eta(\tau,\theta)y_{n}(\tau+\theta)\, d\tau -\int_{s}^{t} \int^{0}_{-r} d_{\theta}\eta(\tau,\theta)y(\tau+\theta)\, d\tau \right|\\
&&\ \ \  \leq  e^{-b^-r} \|y_{n}-y\|_{\mathcal{Y}_{b^-}(\mathbb{R})} \int_{s}^{t} \|L(\tau)\|e^{-b^-|\tau|}\,d\tau \\
&&\ \ \  \leq  e^{-b^-r} \|y_{n}-y\|_{\mathcal{Y}_{b^-}(\mathbb{R})}
\sum_{m=0}^{[\frac{t-s}{r_0}]}\int_{s+([\frac{t-s}{r_0}]-m)r_0}^{s+([\frac{t-s}{r_0}]+1-m)r_0} \|L(\tau)\|e^{-b^-|\tau|}\,d\tau \\
&&\ \ \  \leq  e^{-b^-r} \|y_{n}-y\|_{\mathcal{Y}_{b^-}(\mathbb{R})}
\sum_{m=0}^{[\frac{t-s}{r_0}]}M r_0 e^{-b^-\left\{|s|+([\frac{t-s}{r_0}]+1-m)r_0\right\} }\\
&&\ \ \leq M r_0 e^{-b^-(r+r_0+|s|+t-s)} \|y_{n}-y\|_{\mathcal{Y}_{b^-}(\mathbb{R})} \sum_{m=0}^{[\frac{t-s}{r_0}]} e^{b^- m r_0} \\
&&\ \ \leq \frac{1}{1-e^{b^- r_0}} M r_0 e^{-b^-(r+r_0+|s|+t-s)}\|y_{n}-y\|_{\mathcal{Y}_{b^-}(\mathbb{R})},
\end{eqnarray*}
and
\begin{eqnarray*}
&&\left| \int_{s}^{t} h_n(\tau)\,d\tau-\int_{s}^{t}h(\tau)\,d\tau\right|\\
&&\ \ \leq  \sum_{m=0}^{[\frac{t-s}{r_0}]}e^{-b^-|s+([\frac{t-s}{r_0}]-m)r_0|}\left(e^{b^-|s+([\frac{t-s}{r_0}]-m)r_0|}\int_{s+([\frac{t-s}{r_0}]-m)r_0}^{s+([\frac{t-s}{r_0}]+1-m)r_0}|h_{n}(\tau)-h(\tau)|\,d\tau\right) \\
&&\ \ \leq  e^{-b^-(|s|+t-s)} \|h_{n}-h\|_{\mathcal{H}_{b^-}(\mathbb{R})}\sum_{m=0}^{[\frac{t-s}{r_0}]}e^{b^- m r_0 }\\
&&\ \ \leq  \frac{1}{1-e^{b^- r_0}}e^{-b^-(|s|+t-s)} \|h_{n}-h\|_{\mathcal{H}_{b^-}(\mathbb{R})}.
\end{eqnarray*}
Accordingly, using \eqref{ineq:est-h-y} yields
\[
\begin{aligned}
y(t)-y(s)
&=\lim_{n\to +\infty} (y_{n}(t)-y_{n}(s))\\
&=\lim_{n\to +\infty} \int_{s}^{t} \left(\int^{0}_{-r} d_{\theta}\eta(\tau,\theta)y_{n}(\tau+\theta)+h_{n}(\tau)\right) d\tau \\
&=\int_{s}^{t} \left( \int^{0}_{-r} d_{\theta}\eta(\tau,\theta)y(\tau+\theta)+h(\tau)\right) d\tau,
\end{aligned}
\]
which implies $\mathcal{A}^- h=y$. Hence, $\mathcal{A}^-$ is bounded.
This completes the proof of the lemma.
\end{proof}

{\bf Step 2.} {\it Establish the existence of invariant subspaces and projections.}

To get the decomposition of $C([-r,0],\mathbb{R}^n)$,
we introduce a special solution to equation \eqref{DDEnonhomo}.
For any $\tau\in\mathbb{R}$ and any nonzero function $\phi \in C([-r,0],\mathbb{R}^n)$,
define $x^{\star}:[\tau-r,+\infty)$ as
\begin{eqnarray}\label{df:xstar}
x^{\star}_t:=T(t,\tau)\phi,\ \ \ t\geq \tau.
\end{eqnarray}
By continuity, there exists $\delta_{\phi}>0$
such that $\|x^\star_t\|\neq 0$ for all $t\in [\tau,\tau+\delta_{\phi}]$.
Take any $\delta$ such that  $0<\delta<\min\{1,\delta_{\phi}\}$
and define $h^{\star}:\mathbb{R}\to \mathbb{R}^{n}$ as
\begin{eqnarray}\label{df:hstar}
h^{\star}(t):=
\left\{
\begin{aligned}
 \int^{0}_{-r} d_{\theta}\eta(t,\theta)\left(x^{\star}(t+\theta)\int_{t+\theta}^{t}\frac{1}{\delta_\star}\mathcal{G}_{\star}(u) du\right)
   +\frac{1}{\delta_\star}x^{\star}(t)\mathcal{G}_{\star}(t),
   \ & \ t\in [\tau,+\infty),\\
 0,  \ & \ t\in (-\infty,\tau),
\end{aligned}
\right.
\end{eqnarray}
where
\[
\delta_\star:=\int_{\tau}^{\tau+\delta}\frac{1}{\|x^{\star}_u\|}du,
\ \ \ \ \
\mathcal{G}_{\star}(t):=
\left\{
\begin{aligned}
\frac{1}{\|x^{\star}_t\|}, \ & \ t\in [\tau,\tau+\delta],\\
0, \ & \ t \notin [\tau,\tau+\delta].
\end{aligned}
\right.
\]
Note that
$h^{\star}(t)=0$ for all $t \in (-\infty, \tau) \cup [\tau+\delta+r,+\infty)$.
Then $h^{\star}\in \mathcal{H}_{b^+}(\mathbb{R})$.

Before showing the properly admissible pair associated with $h^{\star}$,
we give the following lemma.

\begin{lemma}\label{lm:uniq}
Suppose that  $(\mathcal{H}_{b^\pm}(\mathbb{R}),\mathcal{Y}_{b^\pm}(\mathbb{R}))$ are properly admissible with respect to \eqref{DDEnonhomo}.
For any $h\in \mathcal{H}_{b^+}(\mathbb{R})$,
if $(h,y^-)$ with $y^-\in \mathcal{Y}_{b^-}(\mathbb{R})$ and $(h,y^+)$ with $y^+ \in \mathcal{Y}_{b^+}(\mathbb{R})$
are two properly admissible pairs for \eqref{DDEnonhomo}, then $y^-=y^+ \in \mathcal{Y}_{b^+}(\mathbb{R})$.
\end{lemma}
\begin{proof}
Since $(h,y^\pm)\in (\mathcal{H}_{b^\pm}(\mathbb{R}),\mathcal{Y}_{b^\pm}(\mathbb{R}))$
are two properly admissible pairs for \eqref{DDEnonhomo},
we have that $y^\pm$ satisfy $\dot y^\pm = L(t)y^\pm_t+h(t)$.
Then $(y^--y^+)$ solves \eqref{DDE}.
Using the fact that $\mathcal{Y}_{b^+}(\mathbb{R})\subset \mathcal{Y}_{b^-}(\mathbb{R})$,
we get $(y^--y^+) \in \mathcal{Y}_{b^-}(\mathbb{R})$.
Accordingly, by the proper admissibility of the pair $(\mathcal{H}_{b^-}(\mathbb{R}),\mathcal{Y}_{b^-}(\mathbb{R}))$ and $0\in \mathcal{H}_{b^-}(\mathbb{R})$,
we have $y^-=y^+$. This completes the proof.
\end{proof}

By the proper admissibility of $(\mathcal{H}_{b^\pm}(\mathbb{R}),\mathcal{Y}_{b^\pm}(\mathbb{R}))$ and Lemma \ref{lm:uniq},
there exists $y^\star \in \mathcal{Y}_{b^+}(\mathbb{R})$ such that
$(h^\star,y^\star)$ is a properly admissible pair with respect to equation \eqref{DDEnonhomo}
in both $(\mathcal{H}_{b^+}(\mathbb{R}),\mathcal{Y}_{b^+}(\mathbb{R}))$ and $(\mathcal{H}_{b^-}(\mathbb{R}),\mathcal{Y}_{b^-}(\mathbb{R}))$.
We next provide some properties of $y^\star$ which are fundamental for the subsequent discussion.

\begin{lemma}\label{lm:imp}
Define $z^\star: [\tau-r,+\infty)\to \mathbb{R}^{n}$ as
\begin{eqnarray}\label{df:zstar}
z^\star(t):=x^\star(t)\int_{-\infty}^{t} \frac{1}{\delta_\star} \mathcal{G}_{\star}(u) du,
\end{eqnarray}
and let $w^\star :=y^\star-z^\star$.
Then
\begin{eqnarray*}
w^\star_t  \!\!\!&=&\!\!\! T(t,\tau)w^\star_\tau=T(t,\tau)y^\star_\tau,
\ \ \ t\geq \tau,\\
y^\star_t  \!\!\!&=&\!\!\! T(t,\tau)(y^\star_\tau+\phi),
~~~~~~~~~~~~~~~~~~~~~~ t\geq \tau+\delta+r.
\end{eqnarray*}
\end{lemma}
\begin{proof}
Fix a time $t\geq \tau$.
Since $y^\star$ is a solution of equation \eqref{DDEnonhomo} with $h=h^{\star}$,
we have
\begin{eqnarray*}
\int_{\tau}^{t}L(\sigma)w^\star_\sigma \, d\sigma
\!\!\! &=&\!\!\! \int_{\tau}^{t}L(\sigma)y^\star_\sigma d\sigma-\int_{\tau}^{t}L(\sigma) z^\star_\sigma d\sigma \\
\!\!\!&=&\!\!\! y^\star(t)-y^\star(\tau)-\int_{\tau}^{t}h^\star(\sigma)d\sigma-\int_{\tau}^{t}L(\sigma) z^\star_\sigma d\sigma \\
\!\!\!&=&\!\!\! y^\star(t)-y^\star(\tau)-\int_{\tau}^{t} \int^{0}_{-r} d_{\theta}\eta(\sigma,\theta) \left( x^\star(\sigma+\theta)\int_{\sigma+\theta}^{\sigma}\frac{1}{\delta_{\star}}\mathcal{G}_{\star}(u) du \right) d\sigma \\
\!\!\!& &\!\!\! - \int_{\tau}^{t}\int^{0}_{-r} d_{\theta}\eta(\sigma,\theta) \left( x^\star(\sigma+\theta)\int_{-\infty}^{\sigma+\theta} \frac{1}{\delta_\star} \mathcal{G}_{\star}(u) du \right) d\sigma
         -\int_{\tau}^{t} \frac{1}{\delta_\star}x^{\star}(\sigma)\mathcal{G}_{\star}(\sigma) d\sigma \\
\!\!\!&=&\!\!\! y^\star(t)-y^\star(\tau)-\int_{\tau}^{t}L(\sigma)x^\star_\sigma\int_{-\infty}^{\sigma} \frac{1}{\delta_\star} \mathcal{G}_{\star}(u) du d\sigma
        -\int_{\tau}^{t} \frac{1}{\delta_\star}x^{\star}(\sigma)\mathcal{G}_{\star}(\sigma) d\sigma.
\end{eqnarray*}
To simplify the last line,
using integration by parts yields
\begin{eqnarray*}
&&\int_{\tau}^{t}L(\sigma)x^\star_\sigma\int_{-\infty}^{\sigma} \frac{1}{\delta_\star} \mathcal{G}_{\star}(u) du d\sigma  \\
&&= \left(\int_{\tau}^{\sigma} L(u)x^\star_u du \int_{-\infty}^{\sigma} \frac{1}{\delta_\star} \mathcal{G}_{\star}(u) du  \right) \Big|_{\sigma=\tau}^{\sigma=t}
 - \int_{\tau}^{t}\left(\int_{\tau}^{\sigma} L(u)x^\star_u du\right) \frac{1}{\delta_\star} \mathcal{G}_{\star}(\sigma) d\sigma \\
\!\!\!\! &&= (x^\star(t)-x^\star(\tau)) \int_{-\infty}^{t} \frac{1}{\delta_\star} \mathcal{G}_{\star}(u) d u
   - \int_{\tau}^{t}\left(x^\star(\sigma)-x^\star(\tau) \right) \frac{1}{\delta_\star} \mathcal{G}_{\star}(\sigma) d\sigma.
\end{eqnarray*}
Accordingly, we get
\begin{eqnarray*}
\int_{\tau}^{t}L(\sigma)w^\star_\sigma \, d\sigma
\!\!\! &=&\!\!\!
y^\star(t)-y^\star(\tau)-(x^\star(t)-x^\star(\tau)) \int_{-\infty}^{t} \frac{1}{\delta_\star} \mathcal{G}_{\star}(u) d u
     - x^\star(\tau)\int_{\tau}^{t}  \frac{1}{\delta_\star}\mathcal{G}_{\star}(u)  d u\\
\!\!\! &=&\!\!\!
y^\star(t)-y^\star(\tau)-x^\star(t)\int_{-\infty}^{t} \frac{1}{\delta_\star}\mathcal{G}_{\star}(u) du
+ x^\star(\tau)\int_{-\infty}^{\tau}  \frac{1}{\delta_\star}\mathcal{G}_{\star}(u) d u \\
\!\!\! &=&\!\!\!
y^\star(t)-y^\star(\tau)-z^\star(t)+z^\star(\tau)\\
\!\!\! &=&\!\!\!
w^\star(t)-w^\star(\tau),
\end{eqnarray*}
which implies that $w^\star$ is a solution of equation \eqref{DDE}.
This together with the fact that  $z^\star(\tau+\theta)=0$ for each $\theta\in [-r,0]$
yields the first assertion.

By the definition of $z^\star$ in \eqref{df:zstar}, we have that
\[
z^\star(t)=x^\star(t),\ \ \ \ \ t\geq \tau+\delta.
\]
Together with the first assertion of this lemma, we get that
\[
y^\star_t=w^\star_t+z^\star_t=T(t,\tau)y^{\star}_\tau +T(t,\tau)\phi=T(t,\tau)(y^\star_\tau+\phi),
\ \ \ t\geq \tau+\delta+r.
\]
This completes the proof.
\end{proof}

For each $\tau \in \mathbb{R}$,
define $\mathcal{S}^\pm(\tau)$ and $\mathcal{U}^\pm(\tau)$ as
\[
\mathcal{S}^\pm(\tau):=\left\{ \phi\in C([-r,0],\mathbb{R}^n):\ \sup_{t\geq \tau}e^{b^\pm|t|}\|T(t, \tau)\phi\|<+\infty \right\},
\]
and
\begin{eqnarray*}
\mathcal{U}^\pm(\tau):=\Biggl\{ \phi\in C([-r,0],\mathbb{R}^n):
\!\!\!&&\!\!\!\!\!
\mbox{ there exists a continuous function } y: (-\infty,\tau]\to \mathbb{R}^{n} \mbox{ such that }\\
\!\!\!&&\!\!\!  y_\tau=\phi,\, y_t=T(t,s)y_s \mbox{ for } s\leq t\leq \tau, \mbox{ and }\sup_{t\leq \tau} e^{b^\pm|t|}\|y_t\|<+\infty \Biggr\}.
\end{eqnarray*}

\begin{lemma}\label{lm:directsum}
$
C([-r,0],\mathbb{R}^n)=\mathcal{S}^\pm(\tau) \oplus \mathcal{U}^\pm(\tau)
$
for each $\tau\in\mathbb{R}$.
Furthermore, $\mathcal{S}^+(\tau)=\mathcal{S}^-(\tau)$ and $\mathcal{U}^+(\tau)=\mathcal{U}^-(\tau)$.
\end{lemma}
\begin{proof}
Fix a time $\tau \in \mathbb{R}$ and take any $\psi \in \mathcal{S}^\pm(\tau) \cap \mathcal{U}^\pm(\tau)$.
Then there exists a continuous function $\tilde{y}:(-\infty,\tau]\to \mathbb{R}^{n}$ such that
$\tilde{y}_\tau=\psi$, $\tilde{y}_t=T(t,s)\tilde{y}_s$ for $s\leq t\leq \tau$ and $\sup_{t\leq \tau}e^{b^\pm|t|}\|\tilde{y}_t\|<+\infty$.
Define $y: \mathbb{R}\to \mathbb{R}^{n}$ as
\[
y_{t}=
\left\{
\begin{aligned}
&\  T(t,\tau)\psi, &  \mbox{ if } t\geq \tau,\\
&\ \tilde{y}_{t},  &   \mbox{ if } t< \tau.
\end{aligned}
\right.
\]
It is clear that  $y$ is continuous and solves equation \eqref{DDE}.
Using the fact that $\psi\in \mathcal{S}^\pm(\tau)$ yields
\[
\sup_{t\geq \tau}e^{b^\pm|t|}\|y_t\|=\sup_{t\geq \tau}e^{b^\pm|t|}\|T(t,\tau)\psi\| <+\infty,
\]
we have
\[
\sup_{t\in \mathbb{R}}e^{b^\pm|t|}|y(t)|
\leq \max\left\{\sup_{t\leq \tau}e^{b^\pm|t|}\|\tilde{y}_t\|,
       \, \sup_{t\geq \tau}e^{b^\pm|t|}\|y_t\|\right\}<+\infty.
\]
This implies $y\in \mathcal{Y}_{b^\pm}(\mathbb{R})$.
Since $0\in \mathcal{H}_{b^\pm}(\mathbb{R})$ and the pair $(\mathcal{H}_{b^\pm}(\mathbb{R}),\mathcal{Y}_{b^\pm}(\mathbb{R}))$ is properly admissible,
we get $y=0$ and then $\psi=0$.
Hence $\mathcal{S}^\pm(\tau) \cap \mathcal{U}^\pm(\tau)=\{0\}$.

For any $\phi\in C([-r,0],\mathbb{R}^n)$,
without loss of generality, assume that $\phi\neq 0$.
Otherwise, $0=\phi \in \mathcal{S}^\pm(\tau)+\mathcal{U}^\pm(\tau)$.
Let the notations be defined as in Lemma \ref{lm:imp}.
Since $h^\star$ defined by \eqref{df:hstar} satisfies
that  $h^\star(t)=0$ for all $t<\tau$,
we have $y^\star_t=T(t,s)y^\star_s \mbox{ for } s\leq t\leq \tau$.
This together with $y^\star\in \mathcal{Y}_{b^\pm}(\mathbb{R})$ yields that
\[
\begin{aligned}
e^{b^\pm |t|}\|y^\star_t\|
& =e^{b^\pm |t|}\sup_{\theta\in [-r,0]}|y^\star(t+\theta)|\\
& \leq e^{b^\pm |t|} \|y^\star\|_{\mathcal{Y}_{b^\pm}(\mathbb{R})}e^{-b^\pm|t|+|b^\pm|r} \\
& \leq e^{|b^\pm| r} \|y^\star\|_{\mathcal{Y}_{b^\pm}(\mathbb{R})} \\
& <+\infty
\end{aligned}
\]
for $t\leq \tau$,
implying that $y^\star_{\tau}\in \mathcal{U}^\pm(\tau)$.
Using the fact that $y^\star \in \mathcal{Y}_{b^\pm}(\mathbb{R})$ and the second assertion of Lemma \ref{lm:imp},
we have $(y^\star_\tau+\phi) \in \mathcal{S}^\pm(\tau)$.
Then
\begin{eqnarray}\label{form:split}
\phi=(y^\star_\tau+\phi)-y^\star_{\tau} \in \mathcal{S}^\pm(\tau)+\mathcal{U}^\pm(\tau).
\end{eqnarray}
Hence we obtain the first assertion of this lemma.

Note that $\mathcal{S}^+(\tau) \subset \mathcal{S}^-(\tau)$ and $\mathcal{U}^+(\tau) \subset \mathcal{U}^-(\tau)$
because $b^-<b^+$. To prove the second assertion, we only need to prove that
$\mathcal{S}^-(\tau) \subset \mathcal{S}^+(\tau)$ and $\mathcal{U}^-(\tau) \subset \mathcal{U}^+(\tau)$.
Consider $\phi \in \mathcal{S}^-(\tau)$.
If $\phi =0$, then $\phi=0 \in \mathcal{S}^+(\tau)$.
If $\phi \neq 0$,
using \eqref{form:split} and $\phi \in \mathcal{S}^-(\tau)$, we get $y^\star_\tau=0$.
Using  \eqref{form:split} again, we obtain
$
\phi=y^\star_\tau+\phi\in \mathcal{S}^+(\tau).
$
Hence  $\mathcal{S}^-(\tau) \subset \mathcal{S}^+(\tau)$.
In contrast,
consider $\phi \in \mathcal{U}^-(\tau)$. If $\phi = 0$,
then $\phi=0 \in \mathcal{U}^+(\tau)$. If  $\phi \neq 0$,
then by \eqref{form:split} we get $y^\star_\tau+\phi=0 \in \mathcal{S}^-(\tau)$.
This together with the inclusion $-y^\star_\tau\in \mathcal{U}^+(\tau)$ yields that $\phi=-y^\star_\tau\in \mathcal{U}^+(\tau)$.
Hence $\mathcal{U}^-(\tau) \subset \mathcal{U}^+(\tau)$.
This completes the proof.
\end{proof}

By Lemma \ref{lm:directsum},
we make the notations
\[
\mathcal{S}(\tau):=\mathcal{S}^+(\tau)=\mathcal{S}^-(\tau),
\qquad
\mathcal{U}(\tau):=\mathcal{U}^+(\tau)=\mathcal{U}^-(\tau),
\]
and let $P(\tau)$ denote the projection on $\mathcal{S}(\tau)$ along $\mathcal{U}(\tau)$.

\begin{lemma}\label{lm:prj}
$\|P(\tau)\|\leq M_1:=r^{-1}_0e^{b^+ (2r+1)}\|\mathcal{A}^+\| (M r_0+1)+1$,
where $\mathcal{A}^+$ is defined in (\ref{AAA}).
\end{lemma}

\begin{proof}
Take any $\phi \in C([-r,0],\mathbb{R}^n)$.
If $\phi \in \mathcal{S}(\tau)$, we have $\|P(\tau)\phi\|=\|\phi\|$.
If $\phi \notin \mathcal{S}(\tau)$,
then $x_t=T(t,\tau)\phi\neq 0$ for all $t\geq \tau$.
Otherwise, if there exists a time $t_0\geq \tau$ such that $x_{t_0}=0$,
then $\phi\in\mathcal{S}(\tau)$.
For this $\phi \notin \mathcal{S}(\tau)$,
let the notations be defined as in the proof of Lemma \ref{lm:imp}
and take any $\delta$ such that $0<\delta<\min\{1,\delta_\phi\}$.
By the proof of Lemma \ref{lm:directsum}, we have
\begin{eqnarray}\label{df:prj}
\phi=P(\tau)\phi=y^{\star}_{\tau}+\phi.
\end{eqnarray}
Noting  $y^{\star}\in \mathcal{Y}_{b^+}(\mathbb{R})$, by Lemma \ref{lm:bddoperator} we get that
\[
\begin{aligned}
\|y^{\star}_{\tau}\|
& \leq \|y^{\star}\|_{\mathcal{Y}_{b^+}(\mathbb{R})}\sup_{\theta\in [-r,0]}e^{-b^+|\tau+\theta|} \\
& \leq e^{b^+ r}\|\mathcal{A}^+\| \|h^{\star}\|_{\mathcal{H}_{b^+}(\mathbb{R})}e^{-b^+ |\tau|}\\
& \leq e^{b^+ r} \|\mathcal{A}^+\|e^{-b^+ |\tau|}
\sup_{s\in \mathbb{R}}\frac{1}{r_0}e^{b^+|s|}\int_{s}^{s+r_0}|h^{\star}(u)| du \\
& \leq \frac{1}{r_0} e^{b^+ r}\|\mathcal{A}^+\|\sup_{\tau \leq s\leq \tau+\delta+r } e^{b^+(|s|-|\tau|)}\int_{s}^{s+r_0}|h^{\star}(u)| du\\
& \leq \frac{1}{r_0} e^{b^+ (2r+\delta)}\|\mathcal{A}^+\| \left( \sup_{\tau \leq s\leq \tau+\delta+r }\int_{s}^{s+r_0}|h^{\star}(u)| du\right).
\end{aligned}
\]
For any $\theta\in [-r,0]$, $u\geq \tau$ and $\sigma\in [u+\theta,u]$,
we have
\[
\sigma-r\leq u+\theta \leq \sigma,
\ \ \
|x^{\star}(u+\theta)| \leq \|x^{\star}_\sigma\|.
\]
It follows that
\[
\begin{aligned}
|x^{\star}(u+\theta)|\int_{u+\theta}^{u}\frac{1}{\delta_\star}|\mathcal{G}_{\star}(\sigma)|d \sigma
& = \frac{1}{\delta_\star}\int_{{\rm supp}(\mathcal{G}_\delta)\cap [u+\theta,u]} |x^{\star}(u+\theta)|\|x^{\star}_\sigma\|^{-1}d\sigma \\
& \leq \frac{1}{\delta_\star}\int_{{\rm supp}(\mathcal{G}_\delta)\cap [u+\theta,u]}d\sigma
  \leq \frac{\delta}{\delta_\star},
\end{aligned}
\]
where we use ${\rm supp}(\mathcal{G}_\delta)\subset [\tau,\tau+\delta]$ in the last inequality.
Then
\begin{eqnarray}\label{est:hstar}
\begin{aligned}
\! \int_{s}^{s+r_0}|h^{\star}(u)| du \!
\leq &\,
\! \int_{s}^{s+r_0}\Big| \int^{0}_{-r} d_{\theta}\eta(u,\theta)\left(x^{\star}(u+\theta)\int_{u+\theta}^{u}\frac{1}{\delta_\star}\mathcal{G}_{\star}(\sigma) d \sigma\right) \Big| du  +\int_{s}^{s+r_0}\frac{1}{\delta_\star}\Big|x^{\star}(u)\mathcal{G}_{\star}(u)\Big|du \\
\leq &\,
\int_{s}^{s+r_0}\|L(u)\|\sup_{-r\leq \theta\leq 0}\left(|x^{\star}(u+\theta)|\int_{u+\theta}^{u}\frac{1}{\delta_\star}|\mathcal{G}_{\star}(\sigma)|d \sigma \right) du
+ \delta \delta_\star^{-1}\\
\leq & \, \delta \delta_\star^{-1}\int_{s}^{s+r_0}\|L(u)\| du +\delta \delta_\star^{-1}.
\end{aligned}
\end{eqnarray}
Accordingly, we get
\[
\begin{aligned}
\|P(\tau)\phi\|
\leq \|y^{\star}_{\tau}\|+\|\phi\|
\leq \frac{1}{r_0} e^{b^+ (2r+\delta)}\|\mathcal{A}^+\| (M r_0+1)\delta \delta_\star^{-1}+\|\phi\|.
\end{aligned}
\]
Note that the left-hand side of the above estimate is independent of $\delta$ and
recall that $\delta^\star$ is defined below \eqref{df:hstar}.
Letting $\delta\to 0$ yields
\[
\begin{aligned}
\|P(\tau)\phi\|
\leq \left(\frac{1}{r_0} e^{b^+ (2r+1)}\|\mathcal{A}^+\| (M r_0+1)+1\right)
\|\phi\|.
\end{aligned}
\]
Therefore, the proof is completed by setting
$M_1=r^{-1}_0e^{b^+ (2r+1)}\|\mathcal{A}^+\| (M r_0+1)+1$.
\end{proof}

\begin{lemma}\label{lm:invariance}
For $t\geq \tau$,
\[
T(t,\tau) \mathcal{S}(\tau)\subset \mathcal{S}(t),
\ \ \
T(t,\tau) \mathcal{U}(\tau)\subset \mathcal{U}(t),
\]
and $T(t,\tau)|_{\mathcal{U}(\tau)}: \mathcal{U}(\tau)\to \mathcal{U}(t)$ is an isomorphism.
\end{lemma}
\begin{proof}
Take $t, \tau\in \mathbb{R}$ such that $t\geq \tau$.
For any $\phi\in \mathcal{S}^\pm(\tau)$, we have
\[
\sup_{s\geq \tau}e^{b^\pm|s|}\|T(s, \tau)\phi\|<+\infty.
\]
It follows that
\[
\sup_{s\geq t}e^{b^\pm|s|}\|T(s, t)(T(t,\tau)\phi)\|
 =\sup_{s\geq t}e^{b^\pm|s|}\|T(s,\tau)\phi)\|
  \leq \sup_{s\geq \tau}e^{b^\pm|s|}\|T(s, \tau)\phi\|<+\infty.
\]
This implies that $T(t,\tau) \mathcal{S}(\tau)\subset \mathcal{S}(t)$.

For any $\phi\in \mathcal{U}^\pm(\tau)$,
there exists $y: (-\infty,\tau]\to \mathbb{R}^{n}$ such that
$y_\tau=\phi$, $y_u=T(u,s)y_s$ for $s\leq u\leq \tau$,
and $\sup_{s\leq \tau}e^{b^\pm|s|}\|y_s\|<+\infty$.
Define $\tilde{y}: \mathbb{R}\to \mathbb{R}^{n}$ by
\begin{eqnarray*}
\tilde{y}_u :=\left\{
\begin{aligned}
&\ y_u, && \mbox{ if } u\leq \tau,\\
&\ T(u,\tau)\phi, && \mbox{ if } u>\tau.
\end{aligned}
\right.
\end{eqnarray*}
It is clear that $\tilde{y}$ is continuous,
$\tilde{y}_{u}=T(u,s)\tilde{y}_{s}$ for $s\leq u\leq t$, and
\begin{eqnarray}\label{est:y-neg}
\sup_{s\leq t}e^{b^\pm|s|}\|\tilde{y}_s\|
=\max\left\{\sup_{s\leq \tau}e^{b^\pm|s|}\|y_s\|, \sup_{\tau\leq s\leq t}e^{b^\pm|s|}\|T(s,\tau)\phi\|\right\}<+\infty.
\end{eqnarray}
Then $\tilde{y}_{t}=T(t,\tau)\phi\in \mathcal{U}^\pm(t)$,
implying $T(t,\tau) \mathcal{U}(\tau)\subset \mathcal{U}(t)$.

For this $\phi \in \mathcal{U}^\pm(\tau)$,
if $T(t,\tau)\phi =0$,
then
\[
\tilde{y}_{u}=T(u,\tau)\phi=T(u,t)T(t,\tau)\phi=0,\ \ \ u\geq t.
\]
This together with \eqref{est:y-neg} implies $\tilde{y}\in \mathcal{H}_{b^\pm}(\mathbb{R})$.
By the proper admissibility of the pair $(0,\tilde{y})$,  we get $\tilde{y}=0$.
Then $\phi=0$ which yields that
$T(t,\tau)|_{\mathcal{U}(\tau)}: \mathcal{U}^\pm(\tau)\to \mathcal{U}^\pm(t)$ is injective.
So is $T(t,\tau)|_{\mathcal{U}(\tau)}: \mathcal{U}(\tau)\to \mathcal{U}(t)$.

Finally, we prove that this map is also surjective.
For any $\hat{\phi}\in\mathcal{U}^\pm(t)$,
there exists $\hat{y}: (-\infty,t]\to \mathbb{R}^{n}$ such that
$\hat{y}_t=\hat{\phi}$, $\hat{y}_u=T(u,s)\hat{y}_s$ for $s\leq u\leq t$,
and $\sup_{s\leq t}e^{b^\pm|s|}\|\hat{y}_s\|<+\infty$.
Particularly, we have
\[
T(t,\tau)\hat{y}_{\tau}=\hat{\phi},
\ \ \
\sup_{s\leq \tau}e^{b^\pm|s|}\|\hat{y}_s\|<+\infty,
\]
implying that $T(t,\tau)|_{\mathcal{U}^\pm(\tau)}: \mathcal{U}^\pm(\tau)\to \mathcal{U}^\pm(t)$ is surjective.
It follows that $T(t,\tau)|_{\mathcal{U}(\tau)}: \mathcal{U}(\tau)\to \mathcal{U}(t)$ is an isomorphism for $t\geq \tau$.
This completes the proof.
\end{proof}

{\bf Step 3.} {\it Estimate the decay rate along the stable subspace.}

By the proper admissibility of the pair $(h^\star,y^\star)$,
we show the decay rate in the stable direction.

\begin{lemma}\label{lm:est:stab}
Let $\beta=\min\{b^+,\,-b^-\}$, $\tilde{\beta}=\max\{b^+,\,-b^-\}$,  and
\begin{eqnarray*}
K_1=\left\{r^{-1}_0(M r_0+1)e^{\tilde{\beta}(1+2r)} \max\left\{\|\mathcal{A}^+\|,\;
\|\mathcal{A}^-\|\right\}+e^{\tilde{\beta} r}\right\}^2\, (M_1+1),
\end{eqnarray*}
where $M_1$ is defined as in Lemma \ref{lm:prj}.
Then for each $\psi \in C([-r,0],\mathbb{R}^n)$,
\[
\|T(t,\tau)P(\tau)\psi\|\leq K_1 e^{-\beta(t-\tau)}\|\psi\|,
\ \ \ t\geq \tau.
\]
\end{lemma}
\begin{proof}
Fix any $\tau\in \mathbb{R}$.
By Lemma \ref{lm:prj}, this lemma holds for $t=\tau$.
For $t>\tau$, we take any nonzero function $\phi \in \mathcal{S}(\tau)$
and $\delta>0$ such that $0<\delta<\min\{t-\tau,\, 1,\, \delta_\phi\}$,
where $\delta_\phi$ satisfies the condition below \eqref{df:xstar}.
Recall that $x^\star$ and $h^\star$ are defined in \eqref{df:xstar} and \eqref{df:hstar}, respectively,
and $(h^\star,y^\star)$ is a properly admissible pair in $(\mathcal{H}_{b^\pm}(\mathbb{R}),\mathcal{Y}_{b^\pm}(\mathbb{R}))$.
We now divide the following proof into three distinct cases according to the signs of $t$ and $\tau$.

{\bf (C1)}
If $t>\tau \geq 0$, noting $y^{\star}\in \mathcal{Y}_{b^+}(\mathbb{R})$, by Lemma \ref{lm:bddoperator} and \eqref{est:hstar} we get
\[
\begin{aligned}
|y^{\star}(t)|
& \leq \|y^{\star}\|_{\mathcal{Y}_{b^+}(\mathbb{R})}e^{-b^+|t|} \\
& \leq \|\mathcal{A}^+\| \|h^{\star}\|_{\mathcal{H}_{b^+}(\mathbb{R})}e^{-b^+ |t|}\\
& \leq \|\mathcal{A}^+\|e^{-b^+ |t|}
  \sup_{s\in \mathbb{R}} \frac{1}{r_0}e^{b^+ |s|}\int_{s}^{s+r_0}|h^{\star}(u)| du \\
& \leq \|\mathcal{A}^+\|\sup_{\tau \leq s\leq \tau+\delta+r }
\frac{1}{r_0}e^{b^+(|s|-|t|)}\int_{s}^{s+r_0}|h^{\star}(u)| du\\
& \leq \frac{1}{r_0}\|\mathcal{A}^+\|(M r_0+1)e^{b^+(1+r)-b^+(t-\tau)}\delta \delta_\star^{-1},
\end{aligned}
\]
where we use $0<\delta<1$ in the last line.
Since $\phi\in\mathcal{S}(\tau)$, by \eqref{df:prj} we have
\[
\phi=P(\tau)\phi=y^\star_\tau+\phi.
\]
Then $y^\star_\tau=0$. This together with the second assertion of Lemma \ref{lm:imp}
yields that $y^\star(t)=x^\star(t)$ for all $t\geq \tau+\delta$,
where $x^\star$ is defined by \eqref{df:xstar}.
Accordingly, we have
\[
|x^\star(t)| \leq  r^{-1}_0\|\mathcal{A}^+\|(M r_0+1)e^{b^+(1+r)-b^+(t-\tau)}\delta \delta_\star^{-1},
\ \ \
t\geq \tau+\delta.
\]
Letting $\delta\to 0$ in the above yields that
\[
|x^\star(t)|
\leq  r^{-1}_0\|\mathcal{A}^+\|(M r_0+1)e^{b^+(1+r)-b^+(t-\tau)}\|\phi\|,
\ \ \
t>\tau\geq 0.
\]
Note that $|x^\star(t+\theta)|\leq \|\phi\|$ if $\theta\in [-r,0]$ and $t+\theta\leq \tau$.
Hence
\begin{eqnarray}\label{est:posit-1}
\|T(t,\tau)\phi\|=\|x^\star_t\|
\leq \left\{r^{-1}_0\|\mathcal{A}^+\|(M r_0+1)e^{b^+(1+2r)}+e^{b^+ r}\right\}  e^{-b^+(t-\tau)}\|\phi\|,
\ \  t> \tau \geq 0.
\end{eqnarray}

{\bf (C2)} If $0\geq t>\tau$, noting $y^{\star}\in \mathcal{Y}_{b^-}(\mathbb{R})$, by Lemma \ref{lm:bddoperator} and \eqref{est:hstar} we have
\[
\begin{aligned}
|y^{\star}(t)|
& \leq \|y^{\star}\|_{\mathcal{Y}_{b^-}(\mathbb{R})}e^{-b^-|t|} \\
& \leq \|\mathcal{A}^-\| \|h^{\star}\|_{\mathcal{H}_{b^-}(\mathbb{R})}e^{-b^- |t|}\\
& \leq \|\mathcal{A}^-\|e^{-b^-|t|}
  \sup_{s\in \mathbb{R}} \frac{1}{r_0}e^{b^- |s|}\int_{s}^{s+r_0}|h^{\star}(u)| du \\
& \leq \|\mathcal{A}^-\|\sup_{\tau \leq s\leq \tau+\delta+r }
\frac{1}{r_0}e^{b^-(|s|-|t|)}\int_{s}^{s+r_0}|h^{\star}(u)| du\\
& \leq \frac{1}{r_0}\|\mathcal{A}^-\|(M r_0+1)e^{b^-(|\tau|-\delta-r-|t|)}\delta \delta_\star^{-1} \\
& \leq \frac{1}{r_0}\|\mathcal{A}^-\|(M r_0+1)e^{-b^-(\delta+r)+b^-(t-\tau)}\delta \delta_\star^{-1}.
\end{aligned}
\]
Recalling $y^\star(t)=x^\star(t)$ for all $t\geq \tau+\delta$, we get
\[
|x^\star(t)| \leq  r^{-1}_0\|\mathcal{A}^-\|(M r_0+1)e^{-b^-(\delta+r)+b^-(t-\tau)}\delta \delta_\star^{-1},
\ \ \
t\geq \tau+\delta.
\]
Accordingly,
letting $\delta\to 0$ in the above yields that
\[
|x^\star(t)|
\leq  r^{-1}_0\|\mathcal{A}^-\|(M r_0+1)e^{-b^-(\delta+r)+b^-(t-\tau)}\|\phi\|,
\qquad 0\geq t>\tau.
\]
It follows that
\begin{eqnarray}\label{est:neg-1}
\|T(t,\tau)\phi\|=\|x^\star_t\|
\leq \left\{r^{-1}_0\|\mathcal{A}^-\|(M r_0+1)e^{-b^-(1+2r)}+e^{-b^- r}\right\}  e^{b^-(t-\tau)}\|\phi\|,
\ \  0\geq t>\tau.
\end{eqnarray}

{\bf (C3)} If $t>0>\tau$, by Lemma \ref{lm:invariance} we obtain $T(0,\tau)\phi \in \mathcal{S}(0)$.
Let
\begin{eqnarray}\label{df:tildK}
\tilde{K}:=r^{-1}_0(M r_0+1)e^{\tilde{\beta}(1+2r)}\, \max\left\{\|\mathcal{A}^+\|,\;
\|\mathcal{A}^-\|\right\}+e^{\tilde{\beta} r},
\end{eqnarray}
where $\tilde{\beta}=\max\{b^+,\,-b^-\}$ and $\beta=\min\{b^+,\,-b^-\}$.
By \eqref{est:posit-1} and \eqref{est:neg-1},
\begin{eqnarray}\label{est:mix-3}
\begin{aligned}
\|T(t,\tau)\phi\|
  & \leq \tilde K  e^{-b^+ t}\|T(0,\tau)\phi\| \\
  & \leq \tilde K e^{-b^+ t}\, \tilde K e^{-b^-\tau}\|\phi\|\\
  & \leq \tilde{K}^2 e^{-\beta (t-\tau)}\|\phi\|,
\ \ \ t>0>\tau.
\end{aligned}
\end{eqnarray}
Hence, using \eqref{est:posit-1},
\eqref{est:neg-1} and \eqref{est:mix-3} yields that
for any $\phi \in \mathcal{S}(\tau)$,
\[
\|T(t,\tau)\phi\|\leq \tilde{K}^2 e^{-\beta (t-\tau)}\|\phi\|,
\ \ \ t\geq \tau.
\]
Therefore, the proof is completed by Lemma \ref{lm:prj} and the above estimates.
\end{proof}

{\bf Step 4.} {\it Estimate the expansion rate along the unstable subspace.}

To obtain the expansion rate of solutions in the unstable subspace,
we introduce another special solution to the nonhomogeneous equation \eqref{DDEnonhomo}.
For any $\tau \in \mathbb{R}$ and any nonzero function $\phi \in \mathcal{U}(\tau)$,
there exists a continuous function $x^* : \mathbb{R} \to \mathbb{R}^n$ such that
$x^*_\tau=\phi$, $x^*_t=T(t,s)x^*_s$ for $t\geq s$
and $\sup_{t\leq \tau}e^{b^+|t|}\|x^*_t\|<+\infty$.
By Lemma \ref{lm:invariance}, this solution $x^*$ is uniquely determined by $x^*_\tau=\phi$.
By continuity, there exists $\delta^\ast_{\phi}>0$ such that $\|x^*_t\|\neq 0$ for all $t\in [\tau-\delta^\ast_{\phi},\tau]$.
Take any constant  $\delta\in (0,\min\{1,\delta^\ast_{\phi}\})$
and define $h^\ast:\mathbb{R}\to \mathbb{R}^{n}$ as
\begin{eqnarray}\label{df:hstar2}
h^{\ast}(t):=
\left\{
\begin{array}{ll}
 -\int^{0}_{-r} d_{\theta}\eta(t,\theta)\left(x^*(t+\theta)\int_{t+\theta}^{t}\frac{1}{\delta_\ast}\mathcal{G}_\ast(u) du\right)
-\frac{1}{\delta^\ast}x^{\ast}(t)\mathcal{G}_\ast(t),
   \ & \ t\in (-\infty,\tau],
   \\
 0,  \ & \ t\in (\tau,+\infty),
\end{array}
\right.
\end{eqnarray}
where
\[
\delta_\ast:=\int^{\tau}_{\tau-\delta}\frac{1}{\|x^{\ast}_u\|}du,
\ \ \ \ \
\mathcal{G}_\ast(t):=
\left\{
\begin{array}{ll}
\frac{1}{\|x^{\ast}_t\|}, \ & \ t\in [\tau-\delta,\tau],\\
0, \ & \ t\notin [\tau-\delta,\tau].
\end{array}
\right.
\]
Since $h^*(t)=0$ for all $t \in (-\infty, \tau-\delta-r) \cup (\tau,+\infty)$,
we have  $h^\ast\in \mathcal{H}_{b^+}(\mathbb{R})$.
By the proper admissibility of $(\mathcal{H}_{b^\pm}(\mathbb{R}),\mathcal{Y}_{b^\pm}(\mathbb{R}))$ and Lemma \ref{lm:uniq},
there exists $y^\ast \in \mathcal{Y}_{b^+}(\mathbb{R})$ such that
$(h^\ast,y^\ast)$ is a properly admissible pair with respect to equation \eqref{DDEnonhomo}
in both
$(\mathcal{H}_{b^+}(\mathbb{R}),\mathcal{Y}_{b^+}(\mathbb{R}))$ and $(\mathcal{H}_{b^-}(\mathbb{R}),\mathcal{Y}_{b^-}(\mathbb{R}))$.

\begin{lemma}\label{lm:imp2}
$y^\ast(t)=x^*(t)$ for all $t \leq \tau-\delta$
and any nonzero $\phi\in\mathcal{U}(\tau)$.
\end{lemma}

\begin{proof}
Define $z^\ast(t):\mathbb{R}\to \mathbb{R}^{n}$ as
\begin{eqnarray}\label{df:zstar2}
z^\ast(t):=x^*(t)\int^{+\infty}_{t} \frac{1}{\delta_\ast} \mathcal{G}_\ast(u) du,
\end{eqnarray}
and let $w^\ast:=y^\ast-z^\ast$.
Since $y^\ast$ is a solution of equation \eqref{DDEnonhomo}, we have
\begin{eqnarray*}
\int^{\tau}_{t}L(\sigma)w^\ast_\sigma \, d\sigma
\!\!\!&=&\!\!\! \int^{\tau}_{t}L(\sigma)y^\ast_\sigma d\sigma-\int^{\tau}_{t}L(\sigma) z^\ast_\sigma d\sigma \\
\!\!\!&=&\!\!\! y^\ast(\tau)-y^\ast(t)-\int^{\tau}_{t}h^\ast(\sigma)d\sigma-\int^{\tau}_{t}L(\sigma) z^\ast_\sigma d\sigma \\
\!\!\!&=&\!\!\! y^\ast(\tau)-y^\ast(t)+\int^{\tau}_{t} \int^{0}_{-r} d_{\theta}\eta(\sigma,\theta) \left( x^*(\sigma+\theta)\int_{\sigma+\theta}^{\sigma}\frac{1}{\delta_\ast}\mathcal{G}_\ast(u) du \right) d\sigma \\
\!\!\!& &\!\!\! - \int^{\tau}_{t}\int^{0}_{-r} d_{\theta}\eta(\sigma,\theta) \left( x^*(\sigma+\theta)\int^{+\infty}_{\sigma+\theta} \frac{1}{\delta_\ast} \mathcal{G}_\ast(u) du \right) d\sigma
         +\int^{\tau}_{t} \frac{1}{\delta_\ast}x^{\ast}(\sigma)\mathcal{G}_\ast(\sigma) d\sigma \\
\!\!\!&=&\!\!\! y^\ast(\tau)-y^\ast(t)-\int^{\tau}_{t}L(\sigma)x^*_\sigma\int^{+\infty}_{\sigma} \frac{1}{\delta_\ast} \mathcal{G}_\ast(u) du d\sigma
        +\int^{\tau}_{t} \frac{1}{\delta_\ast}x^{\ast}(\sigma)\mathcal{G}_\ast(\sigma) d\sigma.
\end{eqnarray*}
To simplify the last line,
using integration by parts, we get
\begin{eqnarray*}
&&-\int^{\tau}_{t}L(\sigma)x^*_\sigma \int^{+\infty}_{\sigma} \frac{1}{\delta_\ast} \mathcal{G}_\ast(u) du d\sigma   \\
&&= \left(\int^{\tau}_{\sigma} L(u)x^*_u du \int^{+\infty}_{\sigma} \frac{1}{\delta_\ast} \mathcal{G}_\ast(u) du  \right) \Big|^{\sigma=\tau}_{\sigma=t}
+\int^{\tau}_{t}\left(\int^{\tau}_{\sigma} L(u)x^*_u du\right) \frac{1}{\delta_\ast} \mathcal{G}_\ast(\sigma) d\sigma \\
&&= (x^*(t)-x^*(\tau)) \int^{+\infty}_{t} \frac{1}{\delta_\ast} \mathcal{G}_\ast(u) d u
    +\int^{\tau}_{t}\left(x^*(\tau)-x^*(\sigma) \right) \frac{1}{\delta_\ast} \mathcal{G}_\ast(\sigma) d\sigma.
\end{eqnarray*}
Accordingly, we obtain
\begin{eqnarray*}
\int^{\tau}_{t}L(\sigma)w^\ast_\sigma \, d\sigma
\!\!\! &=&\!\!\!
y^\ast(\tau)-y^\ast(t)+(x^*(t)-x^*(\tau)) \int^{+\infty}_{t} \frac{1}{\delta_\ast} \mathcal{G}_{\ast}(u) d u
     + x^*(\tau)\int^{\tau}_{t}  \frac{1}{\delta_\ast}\mathcal{G}_{\ast}(u)  d u\\
\!\!\! &=&\!\!\!
y^\ast(\tau)-y^\ast(t)+x^*(t)\int^{+\infty}_{t} \frac{1}{\delta_\ast}\mathcal{G}_{\ast}(u) du
-x^*(\tau)\int^{+\infty}_{\tau}  \frac{1}{\delta_\ast}\mathcal{G}_{\ast}(u) d u \\
\!\!\! &=&\!\!\!
y^\ast(\tau)-y^\ast(t)+z^\ast(t)-z^\ast(\tau)\\
\!\!\! &=&\!\!\!
w^\ast(\tau)-w^\ast(t),
\end{eqnarray*}
which implies that $w^\ast$ is a solution of equation \eqref{DDE}.
Note that $z^\ast(t)=x^*(t)$ for all $t\leq \tau-\delta$ and $z^\ast(t)=0$ for all $t\geq \tau$.
Recalling that $w^\ast=y^\ast-z^\ast$,
$y^*\in \mathcal{Y}_{b^+}(\mathbb{R})$ and $\sup_{t\leq \tau}e^{b^+|t|}\|x^*_t\|<+\infty$, we get $w^* \in \mathcal{Y}_{b^+}(\mathbb{R})$.
The proper admissibility of the pair $(0,w^\ast)$ implies $w^\ast=0$.
Then  $y^\ast(t)=z^\ast(t)=x^*(t)$ for all $t\leq \tau-\delta$ by the definition of $z^\ast$ in \eqref{df:zstar2}.
This completes the proof.
\end{proof}

\begin{lemma}\label{lm:est:unstab}
For each $\psi \in C([-r,0],\mathbb{R}^n)$,
\[
\|T(t,\tau)(Id-P(\tau))\psi\|\leq K_1 e^{-\beta(\tau-t)}\|\psi\|,
\ \ \ t\leq \tau,
\]
where $K_1$ and $\beta$ are defined as in Lemma \ref{lm:est:stab}.
\end{lemma}

\begin{proof}
Fix any $\tau\in \mathbb{R}$.
By Lemma \ref{lm:prj}, this lemma holds for $t=\tau$.
For each $t<\tau$,
we take any nonzero function $\phi \in \mathcal{U}(\tau)$
and any $\delta>0$ such that $0<\delta<\min\{\tau-t,\, 1,\, \delta^*_\phi\}$,
where $\delta^*_\phi$ satisfies the condition before \eqref{df:hstar2}.
Let the notations be defined as in the proof of Lemma \ref{lm:imp2}.
By a similar argument as in the proof of \eqref{est:hstar},
we have
\begin{eqnarray}\label{est:hstar2}
\begin{aligned}
\int_{s}^{s+r_0}|h^*(u)| du
\leq (Mr_0+1)\delta \delta_*^{-1}.
\end{aligned}
\end{eqnarray}
According to the signs of $t$ and $\tau$,
we consider three distinct cases as follows.

{\bf (C1')}
If $0\leq t<\tau$ then, noting  $y^*\in \mathcal{Y}_{b^-}(\mathbb{R})$, by Lemma \ref{lm:bddoperator} and \eqref{est:hstar2}
we have that for any $\theta\in [-r,0]$,
\[
\begin{aligned}
|y^*(t+\theta)|
& \leq \|y^*\|_{\mathcal{Y}_{b^-}(\mathbb{R})}e^{-b^-(|t|+r)} \\
& \leq \|\mathcal{A}^-\| \|h^*\|_{\mathcal{H}_{b^-}(\mathbb{R})}e^{-b^- (|t|+r)}\\
& \leq \|\mathcal{A}^-\|e^{-b^-(|t|+r)}
  \sup_{s\in \mathbb{R}} \frac{1}{r_0}e^{b^- |s|}\int_{s}^{s+r_0}|h^*(u)| du \\
& \leq \|\mathcal{A}^-\|e^{-b^-r}\sup_{\tau-\delta-r \leq s\leq \tau}
\frac{1}{r_0}e^{b^-(|s|-|t|)}\int_{s}^{s+r_0}|h^*(u)| du\\
& \leq \frac{1}{r_0}\|\mathcal{A}^-\|(M r_0+1)e^{-b^- (1+2r)+b^-(\tau-t)}\delta \delta_*^{-1},
\end{aligned}
\]
where we use $0<\delta<1$ in the last line.
By Lemma \ref{lm:imp2}, we have that
\[
\|x^*_t\|
\leq \frac{1}{r_0}\|\mathcal{A}^-\|(M r_0+1)e^{-b^- (1+2r) + b^-(\tau-t)}\delta \delta_*^{-1},
\ \ \ t\leq \tau-\delta.
\]
Letting $\delta\to 0$ in the above, by Lemma \ref{lm:invariance} we see
\begin{eqnarray}\label{est:posit2}
\|T(t,\tau)\phi\|=\|x^*_t\|
\leq  r^{-1}_0 \|\mathcal{A}^-\|(M r_0+1)e^{-b^- (1+2r) + b^-(\tau-t)} \|\phi\|,
\ \ \ 0\leq t<\tau.
\end{eqnarray}

{\bf (C2')} If $t<\tau \leq 0$, noting  $y^*\in \mathcal{Y}_{b^+}(\mathbb{R})$, by Lemma \ref{lm:bddoperator} and \eqref{est:hstar2} we get that for any $\theta\in [-r,0]$,
\[
\begin{aligned}
|y^*(t+\theta)|
& \leq \|y^*\|_{\mathcal{Y}_{b^+}(\mathbb{R})}e^{-b^+(|t|-r)} \\
& \leq \|\mathcal{A}^+\| \|h^*\|_{\mathcal{H}_{b^+}(\mathbb{R})}e^{-b^+ (|t|-r)}\\
& \leq \|\mathcal{A}^+\|e^{-b^+(|t|-r)}
  \sup_{s\in \mathbb{R}} \frac{1}{r_0}e^{b^+ |s|}\int_{s}^{s+r_0}|h^*(u)| du \\
& \leq \|\mathcal{A}^+\|e^{b^+ r}\sup_{\tau-\delta-r\leq s\leq \tau}
\frac{1}{r_0}e^{b^+(|s|-|t|)}\int_{s}^{s+r_0}|h^*(u)| du\\
& \leq \frac{1}{r_0}\|\mathcal{A}^+\|(M r_0+1)e^{b^+(|\tau|+\delta+2r-|t|)}\delta \delta_*^{-1} \\
& \leq \frac{1}{r_0}\|\mathcal{A}^+\|(M r_0+1)e^{b^+(1+2r)-b^+(\tau-t)}\delta \delta_*^{-1}.
\end{aligned}
\]
Using a similar argument to that in the proof of \eqref{est:posit2},
we have
\begin{eqnarray}\label{est:neg-2}
\|T(t,\tau)\phi\|
  \leq r_0^{-1} \|\mathcal{A}^+\|(M r_0+1)e^{b^+(1+2r)-b^+(\tau-t)}\|\phi\|,
  \ \ \ t<\tau\leq 0.
\end{eqnarray}

{\bf (C3')} If $t<0<\tau$, by Lemma \ref{lm:invariance} we have $T(0,\tau)\phi \in \mathcal{U}(0)$.
Recall that $\tilde{K}$ is defined by \eqref{df:tildK}.
Using \eqref{est:posit2} and \eqref{est:neg-2} yields that
\begin{eqnarray}\label{est:mix-2}
\begin{aligned}
\|T(t,\tau)\phi\|
  &=\|T(t,0)T(0,\tau)\phi\| \\
  &\leq \tilde K  e^{b^+ t}\|T(0,\tau)\phi\| \\
  &\leq \tilde K e^{b^+ t}\cdot \tilde K e^{b^-\tau}\|\phi\| \\
  &\leq \tilde{K}^2 e^{-\beta (\tau-t)}\|\phi\|
\end{aligned}
\end{eqnarray}
for $t<0<\tau$. Hence, by \eqref{est:posit2}, \eqref{est:neg-2} and \eqref{est:mix-2}
we see that for any $\phi \in \mathcal{U}(\tau)$,
\[
\|T(t,\tau)\phi\|\leq \tilde{K}^2 e^{-\beta (t-\tau)}\|\phi\|,
\ \ \  t \leq \tau.
\]
Therefore, the proof is completed by Lemma \ref{lm:prj} and the above estimates.	
\end{proof}

Finally, we proceed to the proof of Proposition \ref{prop-adm-to-NED}.

\begin{proof}[\itshape\bfseries Proof of Proposition \ref{prop-adm-to-NED}]
For each \( t \in \mathbb{R} \), let \( P(t) \) be defined as in Lemma~\ref{lm:prj}.
By Lemma~\ref{lm:invariance}, the evolution family \( \{ T(t,s) : t \ge s \} \) on \( C([-r,0], \mathbb{R}^n) \) satisfies {\bf (E1)} and {\bf (E2)}.
Let \( \beta \) and \( K_1 \) be the constants defined as in Lemma~\ref{lm:est:stab}.
In view of Lemmas~\ref{lm:est:stab} and~\ref{lm:est:unstab},
we get that {\bf (E3)} also holds with \( \alpha = \beta \) and \( K = K_1 \).
Consequently, equation \eqref{DDE} admits an exponential dichotomy with exponent \( \beta \) and bound \( K_1 \).
This completes the proof.
\end{proof}


\section{Robustness via the admissibility property}\
\label{sec:pfs}

In this section,
we give the main results on the robustness of exponential dichotomies against small-delay perturbations.
Therefore, we assume that delay equation \eqref{DDE} satisfies {\bf (A4)}.
The presence of delays transforms the phase space from $\mathbb{R}^{n}$ into $C([-r,0], \mathbb{R}^n)$,
which causes that the analysis framework developed in \cite{BV-20b,Elorreaga-Gomez-25}, where perturbations preserve the phase space, are now inapplicable.
We address this issue through the admissible characterization established in the preceding and an operator perturbation argument.


\subsection{Operator perturbation}\

For any fixed $b\in \mathbb{R}$,
define the linear operator $\mathcal{T}$ on the domain
$
\mathcal{D}(\mathcal{T}):=\Big\{ y\in \mathcal{Y}_{b}(\mathbb{R}):\;  \dot y(t)-A(t)y(t) \in \mathcal{H}_{b}(\mathbb{R})
 \Big\}
$
as
\begin{eqnarray*}
(\mathcal{T}y)(t):=\dot y(t)-A(t)y(t), \qquad t\in \mathbb{R}.
\end{eqnarray*}
This operator is the one associated with ODE \eqref{eq:ODE}.
Accordingly,
for its small-delay perturbation, i.e., equation \eqref{DDE} having small delay and satisfying {\bf (A4)},
we define the linear operator $\tilde{\mathcal{T}}$
on the domain
$
\mathcal{D}(\tilde{\mathcal{T}})
:=\Big\{ y\in \mathcal{Y}_{b}(\mathbb{R}):\;  \dot y(t)-L(t)y_{t}\in \mathcal{H}_{b}(\mathbb{R}) \Big\}
$
as
\begin{eqnarray*}
(\tilde{\mathcal{T}}y)(t):=\dot y(t)-L(t)y_{t}, \qquad t\in \mathbb{R}.
\end{eqnarray*}
Using the same argument as in page 108 of \cite{CL-99} or pp.~22-23 of \cite{Coppel-78},
we see that  $\mathcal{T}:\mathcal{D}(\mathcal{T})\subset \mathcal{Y}_{b}(\mathbb{R})\to \mathcal{H}_{b}(\mathbb{R})$
and $\tilde{\mathcal{T}}:\mathcal{D}(\tilde{\mathcal{T}})\subset \mathcal{Y}_{b}(\mathbb{R})\to \mathcal{H}_{b}(\mathbb{R})$ are closed.

We now study the properties of $\tilde{\mathcal{T}}$ via an operator perturbation approach.
Precisely, following the method used by Hale and Verduyn Lunel \cite[Section 5.4]{Hale-Lunel-01},
we write $(\tilde{\mathcal{T}}y)(t)$ as
\begin{eqnarray}\label{df:T-pert}
(\tilde{\mathcal{T}}y)(t)=(\mathcal{T}y)(t)+(\Lambda y)(t),
\end{eqnarray}
where
\begin{eqnarray*}
(\Lambda y)(t)=A(t)y(t)-\int^{0}_{-r} d_\theta\eta(t,\theta)y(t+\theta).
\end{eqnarray*}
We regard $\mathcal{T}$ as the unperturbed operator and $\Lambda$ the perturbation operator.
For each $y\in \mathcal{D}(\tilde{\mathcal{T}})$, using integration by parts, we further transform $(\Lambda y)(t)$ into the following:
\begin{eqnarray}\label{df:pert}
\begin{aligned}
(\Lambda y)(t)
 &=A(t)y(t)-\int^{0}_{-r} d_\theta\eta(t,\theta)y(t+\theta)\\
 &=A(t)y(t)-A(t)y(t-r)+\int^{0}_{-r} \eta(t,\theta)\dot y(t+\theta)\,d\theta\\
 &=A(t)\int_{t-r}^{t}\dot y(\theta)\,d\theta+\int^{t}_{t-r} \eta(t,\theta-t)\dot y(\theta)\,d\theta.
\end{aligned}
\end{eqnarray}
This formula will play an important role in the subsequent discussion.

We first show that the domain $\mathcal{D}(\mathcal{T})$ is invariant under the perturbation $\Lambda$.

\begin{lemma}\label{lm:domain-T-Tr}
Suppose that equation \eqref{DDE} satisfies {\bf (A4)}.
Then $\mathcal{D}(\tilde{\mathcal{T}})=\mathcal{D}(\mathcal{T})$ for any $b\in \mathbb{R}$.
\end{lemma}
\begin{proof}
Fix any $b\in \mathbb{R}$.
For any $y\in \mathcal{D}(\tilde{\mathcal{T}})$, there exists $h\in \mathcal{H}_{b}(\mathbb{R})$ such that
\begin{eqnarray}\label{eq:y-h}
\dot y(t)=L(t)y_{t}+h(t).
\end{eqnarray}
Together with \eqref{df:T-pert}, it implies that
\[
(\mathcal{T}y)(t)=h(t)-(\Lambda y)(t)=:\tilde{h}(t).
\]
Substituting \eqref{eq:y-h} in \eqref{df:pert} yields
\[
(\Lambda y)(t)
=A(t)\int_{t-r}^{t}\left(L(u)y_{u}+h(u)\right)\,d u+\int_{t-r}^{t}\eta(t,u-t)\left(L(u)y_{u}+h(u)\right)\,d u.
\]
Noting \eqref{hyp-int} and $h\in \mathcal{H}_{b}(\mathbb{R})$, we have
\begin{eqnarray*}
\int_{t-r}^{t}\Big| L(u)y_{u}+h(u) \Big|\,d u
\!\!\! &\leq&\!\!\!
\int_{t-r_0}^{t}Me^{|b|r_0}\|y\|_{\mathcal{Y}_{b}(\mathbb{R})}e^{-b|u|}\, d u + \int_{t-r_0}^{t} |h(u)| \,d u \\
\!\!\! &\leq&\!\!\!
r_0\left(M  e^{2|b|r_0} \|y\|_{\mathcal{Y}_{b}(\mathbb{R})}+ e^{|b|r_0}\|h\|_{\mathcal{H}_{b}(\mathbb{R})}\right)e^{-b|t|}.
\end{eqnarray*}
Accordingly, we have the following estimates
\begin{eqnarray*}
\frac{1}{r_0}e^{b |t|}\int_{t}^{t+r_0}|(\Lambda  y)(\varsigma)|\,d \varsigma
\!\!\!&\leq&\!\!\!
\frac{1}{r_0}e^{b |t|}\int_{t}^{t+r_0}\int_{\varsigma-r}^{\varsigma}\Big|\Big(A(\varsigma)+\eta(\varsigma,u-\varsigma)\Big)\Big( L(u)y_{u}+h(u) \Big)\Big|
\,d u \,d \varsigma \\
\!\!\!&\leq&\!\!\!
M \left(M  e^{2|b|r_0} \|y\|_{\mathcal{Y}_{b}(\mathbb{R})}+ e^{|b|r_0}\|h\|_{\mathcal{H}_{b}(\mathbb{R})}\right)
     \int_{t}^{t+r_0} e^{b(|t|-|\varsigma|)}\, d \varsigma  \\
\!\!\!&\leq&\!\!\!
r_0 M e^{|b|r_0}\left(M  e^{2|b|r_0} \|y\|_{\mathcal{Y}_{b}(\mathbb{R})}+ e^{|b|r_0}\|h\|_{\mathcal{H}_{b}(\mathbb{R})}\right).
\end{eqnarray*}
This implies $\Lambda y \in \mathcal{H}_{b}(\mathbb{R})$, and thereby, $\tilde{h}\in \mathcal{H}_{b}(\mathbb{R})$ and $y\in \mathcal{D}(\mathcal{T})$.
Hence $\mathcal{D}(\tilde{\mathcal{T}})\subset \mathcal{D}(\mathcal{T})$.
Similarly, we can prove that $\mathcal{D}(\mathcal{T})\subset \mathcal{D}(\tilde{\mathcal{T}})$.
This completes the proof.
\end{proof}

Next we provide a detailed description of the domain $\mathcal{D}(\tilde{\mathcal{T}})$ (or $\mathcal{D}(\mathcal{T})$ by Lemma \ref{lm:domain-T-Tr})
in the following two lemmas,
provided that ODE \eqref{eq:ODE} admits an exponential dichotomy.

\begin{lemma}\label{lm:F-adm}
Suppose that ODE \eqref{eq:ODE} admits an exponential dichotomy
$\mathcal{E}(\alpha,K)$ on  $\mathbb{R}$ with the dichotomy projection $P(t)$ for each $t\in\mathbb{R}$.
Fix any $b\in (-\alpha,\alpha)$ and define a map $\mathcal{F}$ on $\mathcal{H}_{b}(\mathbb{R})$  as
\begin{eqnarray}\label{df:F}
(\mathcal{F} h)(t)
:=\int_{-\infty}^{t}\Psi(t,s)P (s)h(s)\,ds-\int^{+\infty}_{t}\Psi(t,s)(Id-P(s))h(s)\,ds,
\ \ \ h\in \mathcal{H}_{b}(\mathbb{R}),
\end{eqnarray}
where $\Psi(t,s)$ denotes the evolution operator of ODE \eqref{eq:ODE}, i.e.,
$\Psi(\cdot,s)\xi$ is the unique solution of  \eqref{eq:ODE} with initial condition $x(s) = \xi$
for $s\in \mathbb{R}$ and $\xi\in\mathbb{R}^n$. Then
$\mathcal{F}$ is a continuous linear operator from $\mathcal{H}_{b}(\mathbb{R})$ to $\mathcal{Y}_{b}(\mathbb{R})$ such that
\[
\|\mathcal{F}h\|_{\mathcal{Y}_{b}(\mathbb{R})}
\leq  \frac{4K r_0 e^{(\alpha+|b|) r_0}}{1-e^{-(\alpha-|b|)r_0}}\|h\|_{\mathcal{H}_{b}(\mathbb{R})}, \qquad  h\in\mathcal{H}_{b}(\mathbb{R}).
\]
\end{lemma}
\begin{proof}
Suppose that ODE \eqref{eq:ODE} admits an exponential dichotomy
$\mathcal{E}(\alpha,K)$ on  $\mathbb{R}$ with the dichotomy projection $P(t)$ for each $t\in\mathbb{R}$.
For each $h\in\mathcal{H}_{b}(\mathbb{R})$, let
\begin{eqnarray}\label{df:F-1-2}
\begin{aligned}
(\mathcal{F}_{1}h)(t) &:= \int_{-\infty}^{t}\Psi(t,s)P(s)h(s)\,ds, \\
(\mathcal{F}_{2}h)(t) &:= -\int^{+\infty}_{t}\Psi(t,s)(Id-P(s))h(s)\,ds.
\end{aligned}
\end{eqnarray}
Then
\[
(\mathcal{F}h)(t)=(\mathcal{F}_{1}h)(t)+(\mathcal{F}_{2}h)(t), \ \ \ \ t\in\mathbb{R}.
\]
Since $P(t)$ is continuous in $t$,
we have that  $\mathcal{F}h$, $\mathcal{F}_{1}h$ and $\mathcal{F}_{2}h$ are continuous functions on $\mathbb{R}$ and linear in $h$.

Fix any $h\in \mathcal{H}_{b}(\mathbb{R})$. Since $\alpha+b>0$,
we have that for $t\leq 0$,
\begin{eqnarray}\label{est:F1h-neg0}
\begin{aligned}
|(\mathcal{F}_{1}h)(t)|
&\leq  \int_{-\infty}^{t}Ke^{-\alpha(t-s)}|h(s)|\,ds\\
& \leq K \sum_{m=0}^{\infty}\int^{t-m r_0}_{t-(m+1)r_0} e^{-\alpha m r_0}|h(s)|\,ds\\
& = K r_0  e^{-b r_0} e^{b t}\sum_{m=0}^{\infty}e^{-(\alpha+b) mr_0} \left(\frac{1}{r_0}e^{b|t-(m+1)r_0|}\int^{t-m r_0}_{t-(m+1)r_0} |h(s)|\,ds \right)\\
&\leq \frac{K r_0  e^{-b r_0}}{1-e^{-(\alpha+b) r_0}}\|h\|_{\mathcal{H}_{b}(\mathbb{R})}e^{b t}.
\end{aligned}
\end{eqnarray}
For $t>0$, noting \eqref{est:F1h-neg0} and $t-r_0<[\frac{t}{r_0}] r_0\leq t$, we can compute
\[
\begin{aligned}
|(\mathcal{F}_{1}h)(t)|
&\leq  \int_{-\infty}^{0}Ke^{-\alpha(t-s)}|h(s)|\,ds +\int_{0}^{t}Ke^{-\alpha(t-s)}|h(s)|\,ds\\
&\leq  |\mathcal{F}_{1}h(0)| e^{-\alpha t}
  +K\sum_{m=0}^{[\frac{t}{r_0}]} \int_{([\frac{t}{r_0}]-m)r_0}^{([\frac{t}{r_0}]+1-m)r_0} e^{-\alpha ([\frac{t}{r_0}]r_0-s)}|h(s)|\,ds \\
&\leq  \frac{K r_0  e^{-b r_0}}{1-e^{-(\alpha+b) r_0}}\|h\|_{\mathcal{H}_{b}(\mathbb{R})} e^{-\alpha t}
  +K\sum_{m=0}^{[\frac{t}{r_0}]} e^{-\alpha(m-1)r_0} \int_{([\frac{t}{r_0}]-m)r_0}^{([\frac{t}{r_0}]+1-m)r_0}|h(s)|\,ds.
\end{aligned}
\]
Since $h\in\mathcal{H}_{b}(\mathbb{R})$ and $b-\alpha<0$, we have the following estimates on the second term:
\[
\begin{aligned}
& K\sum_{m=0}^{[\frac{t}{r_0}]} e^{-\alpha(m-1)r_0} \int_{([\frac{t}{r_0}]-m)r_0}^{([\frac{t}{r_0}]+1-m)r_0}|h(s)|\,ds\\
& \leq   K r_0 e^{(\alpha+|b|) r_0} \sum_{m=0}^{[\frac{t}{r_0}]} e^{-(\alpha-b)mr_0-b t} \left( \frac{1}{r_0}e^{b([\frac{t}{r_0}]-m)r_0} \int_{([\frac{t}{r_0}]-m)r_0}^{([\frac{t}{r_0}]+1-m)r_0}|h(s)|\,ds \right)\\
&\leq  \frac{K r_0 e^{(\alpha+|b|) r_0} }{1-e^{-(\alpha-b)r_0}}\|h\|_{\mathcal{H}_{b}(\mathbb{R})}e^{-bt},
\ \ \ \ t>0.
\end{aligned}
\]
Accordingly, we have
\[
|(\mathcal{F}_{1}h)(t)|
\leq \frac{2K r_0 e^{(\alpha+|b|) r_0} }{1-e^{-(\alpha-|b|)r_0}}\|h\|_{\mathcal{H}_{b}(\mathbb{R})}e^{-bt}
\ \ \ \mbox{ for }  t>0.
\]
Together with \eqref{est:F1h-neg0}, it implies that
\begin{eqnarray}\label{est:F1h-R2}
\begin{aligned}
\|\mathcal{F}_{1}h\|_{\mathcal{Y}_{b}(\mathbb{R})}
\leq \frac{2K r_0 e^{(\alpha+|b|) r_0} }{1-e^{-(\alpha-|b|)r_0}}\|h\|_{\mathcal{H}_{b}(\mathbb{R})}.
\end{aligned}
\end{eqnarray}

As for  $\mathcal{F}_{2}h$,  we have that for $t\geq 0$,
\begin{eqnarray}\label{est:F2h-pos2}
\begin{aligned}
|(\mathcal{F}_{2}h)(t)|
& \leq  \int^{\infty}_{t}Ke^{-\alpha(s-t)}|h(s)|\,ds\\
& \leq K \sum_{m=0}^{\infty}\int_{t+m r_0}^{t+(m+1)r_0} e^{-\alpha m r_0}|h(s)|\,ds\\
& = K r_0 e^{-b t}\sum_{m=0}^{\infty} e^{-(\alpha+b) m r_0} \left( \frac{1}{r_0}e^{b(t+m r_0)}\int^{t+(m+1)r_0}_{t+m r_0}|h(s)|\,ds \right)\\
& \leq \frac{K r_0}{1-e^{-(\alpha+b)r_0}}\|h\|_{\mathcal{H}_{b}(\mathbb{R})}e^{-b t},
\end{aligned}
\end{eqnarray}
where we note that $\alpha+b>0$ and $h\in \mathcal{H}_{b}(\mathbb{R})$.
For $t<0$, using \eqref{est:F2h-pos2} yields
\[
\begin{aligned}
|(\mathcal{F}_{2}h)(t)|
&\leq  \int^{\infty}_{0}Ke^{-\alpha(s-t)}|h(s)|\,ds+\int^{0}_{t}Ke^{-\alpha(s-t)}|h(s)|\,ds\\
&\leq \frac{K r_0}{1-e^{-(\alpha+b)r_0}}\|h\|_{\mathcal{H}_{b}(\mathbb{R})}e^{\alpha t}
  +K\sum_{m=0}^{[-\frac{t}{r_0}]}\int_{(m-1-[-\frac{t}{r_0}])r_0}^{(m-[-\frac{t}{r_0}])r_0}e^{-\alpha(s+[-\frac{t}{r_0}]r_0)}|h(s)|\,ds,
\end{aligned}
\]
where we use the fact that $t\leq -[-\frac{t}{r_0}]r_0<t+r_0$ in the last line.
We further have the following
\[
\begin{aligned}
& K\sum_{m=0}^{[-\frac{t}{r_0}]}\int_{(m-1-[-\frac{t}{r_0}])r_0}^{(m-[-\frac{t}{r_0}])r_0}e^{-\alpha(s+[-\frac{t}{r_0}]r_0)}|h(s)|\,ds\\
&\ \ \ \leq K\sum_{m=0}^{[-\frac{t}{r_0}]}e^{-\alpha(m-1)r_0}\int_{(m-1-[-\frac{t}{r_0}])r_0}^{(m-[-\frac{t}{r_0}])r_0}|h(s)|\,ds \\
&\ \ \ = K r_0 e^{(\alpha-b) r_0-b [-\frac{t}{r_0}]r_0}\sum_{m=0}^{[-\frac{t}{r_0}]}e^{-(\alpha-b) m r_0}\left( \frac{1}{r_0}e^{b([-\frac{t}{r_0}]-m+1)r_0}\int_{(m-1-[-\frac{t}{r_0}])r_0}^{(m-[-\frac{t}{r_0}])r_0}|h(s)|\,ds \right) \\
&\ \ \ \leq \frac{K r_0 e^{(\alpha+|b|) r_0}}{1-e^{-(\alpha-b)r_0}}\|h\|_{\mathcal{H}_{b}(\mathbb{R})}e^{b t}, \ \ \ \ t< 0.
\end{aligned}
\]
It follows that
\[
\begin{aligned}
|(\mathcal{F}_{2}h)(t)|
\leq \frac{2K r_0 e^{(\alpha+|b|) r_0}}{1-e^{-(\alpha-|b|)r_0}}\|h\|_{\mathcal{H}_{b}(\mathbb{R})}e^{b t}, \ \ \ \ t< 0.
\end{aligned}
\]
Together with \eqref{est:F2h-pos2}, it implies that
\begin{eqnarray}\label{est:F2h-R2}
\begin{aligned}
\|\mathcal{F}_{2}h\|_{\mathcal{Y}_{b}(\mathbb{R})}
\leq  \frac{2K r_0 e^{(\alpha+|b|) r_0}}{1-e^{-(\alpha-|b|)r_0}}\|h\|_{\mathcal{H}_{b}(\mathbb{R})}.
\end{aligned}
\end{eqnarray}
Therefore, this lemma is proved by \eqref{est:F1h-R2} and \eqref{est:F2h-R2}.
\end{proof}

\begin{lemma}\label{lm:EDAdmis-ODE}
Suppose that ODE \eqref{eq:ODE} admits an exponential dichotomy $\mathcal{E}(\alpha,K)$ on $\mathbb{R}$.
Fix any $b\in (-\alpha,\alpha)$ and let $\mathcal{F}:\mathcal{H}_{b}(\mathbb{R})\to \mathcal{Y}_{b}(\mathbb{R})$ be defined by \eqref{df:F}.
Then there exists a unique $y\in \mathcal{Y}_{b}(\mathbb{R})$ such that $\mathcal{T}y=h$ for each $h\in \mathcal{H}_{b}(\mathbb{R})$.
Furthermore,
$\mathcal{D}(\tilde{\mathcal{T}})=\mathcal{D}(\mathcal{T})=\mathcal{R}(\mathcal{F})$,
where $\mathcal{R}(\mathcal{F})$ is the range of $\mathcal{F}$.
\end{lemma}
\begin{proof}
Suppose that ODE \eqref{eq:ODE} admits an exponential dichotomy $\mathcal{E}(\alpha,K)$ on $\mathbb{R}$.
Fix any $b\in (-\alpha,\alpha)$.
It suffices to prove that $x=0$ if $x(t)=\Psi(t,s)x(s)$ for $t\geq s$ such that $x\in \mathcal{Y}_{b}(\mathbb{R})$.
We still use $P(t)$ to denote the dichotomy projection of the exponential dichotomy $\mathcal{E}(\alpha,K)$ for ODE \eqref{eq:ODE}.
Fix any $t\in \mathbb{R}$ and let
\[
v(t)=P(t)x(t),\ \ \ \ u(t)=(Id-P(t))x(t).
\]
Then for any $s\geq 0$,
\[
v(t)=P(t)\Psi(t,t-s)x(t-s)=\Psi(t,t-s)P(t-s)v(t-s).
\]
Note that $v\in \mathcal{Y}_{b}(\mathbb{R})$ because  $x\in \mathcal{Y}_{b}(\mathbb{R})$. Then
\[
|v(t)|\leq Ke^{-\alpha s}\|v\|_{\mathcal{Y}_{b}(\mathbb{R})}e^{-b|t-s|}\leq K\|v\|_{\mathcal{Y}_{b}(\mathbb{R})}e^{b t}e^{-(\alpha+b)s}
\]
for all $s\geq \max\{0,t\}$. Using the fact that $\alpha+b>0$ and letting $s\to +\infty$, we get $v(t)=0$.
Similarly, we have $u(t)=0$. Consequently, we get $x=0$.
This proves the uniqueness.

For any $h\in \mathcal{H}_{b}(\mathbb{R})$, we have $\mathcal{F}h\in \mathcal{Y}_{b}(\mathbb{R})$.
By \eqref{df:F}, we can verify that $\mathcal{F}h$ is continuously differentiable and
\[
\mathcal{T}(\mathcal{F}h)(t)= \frac{d}{dt} (\mathcal{F}h)(t)-A(t) (\mathcal{F}h)(t)=h(t),\ \ \
\forall t\in\mathbb{R}.
\]
This implies $\mathcal{F}h\in \mathcal{D}(\mathcal{T})$.
Noting $\mathcal{T}(\mathcal{F}h)=h$ and the first assertion of this lemma,
we get $\mathcal{D}(\mathcal{T})=\mathcal{R}(\mathcal{F})$.
Therefore, the proof is completed by Lemma \ref{lm:domain-T-Tr}.
\end{proof}

For each $y\in \mathcal{D}(\tilde{\mathcal{T}})$,
by Lemma \ref{lm:EDAdmis-ODE},
there exists a unique $h\in \mathcal{H}_{b}(\mathbb{R})$ such that $\mathcal{F}h=y$.
Then   \eqref{df:T-pert} is equivalent to the form
\begin{eqnarray}\label{df:TrF}
(\tilde{\mathcal{T}}\circ \mathcal{F}h)(t)=(\mathcal{T}\circ \mathcal{F} h)(t)+(\Lambda \circ \mathcal{F} h)(t).
\end{eqnarray}
We will use this equivalent form to prove the admissibility property of equation \eqref{DDE} with small delay.
With respect to the three composite operators appearing in this formula, we have the next lemma.

\begin{lemma}\label{lm:fdmentlm}
Suppose that ODE \eqref{eq:ODE} admits an exponential dichotomy
$\mathcal{E}(\alpha, K)$ on $\mathbb{R}$ and equation \eqref{DDE} satisfies {\bf (A1)}\,-\,{\bf (A4)}. Then we have the following assertions:
\begin{enumerate}
\item[{\bf (i)}]
For each $b\in (-\alpha,\alpha)$,
the operator $\mathcal{T}\circ \mathcal{F}$ is the identity from $\mathcal{H}_{b}(\mathbb{R})$
to $\mathcal{H}_{b}(\mathbb{R})$.
\vskip 3pt

\item[{\bf (ii)}]
For each $b\in (-\alpha,\alpha)$ and $r\in (0,r_0]$,
the operator $\Lambda \circ \mathcal{F}$
is a bounded linear operator from $\mathcal{H}_{b}(\mathbb{R})$ to itself and
\[
\|\Lambda \circ \mathcal{F}h\|_{\mathcal{H}_{b}(\mathbb{R})}
\leq r_0 \left\{\frac{4KM^2 r_0 e^{4\alpha r_0}}{1-e^{-(\alpha-|b|)r_0}}+Me^{2\alpha r_0}\right\}\|h\|_{\mathcal{H}_{b}(\mathbb{R})}.
\]
\vskip 3pt

\item[{\bf (iii)}]
For each $r\in (0,r_0]$,
where
$r_0>0$ is a constant such that the function
\begin{eqnarray}\label{df:Kar0}
\ell(\alpha,r_0):=Me^{4\alpha r_0}(4KM+\alpha)
\end{eqnarray}
satisfies
$\ell(\alpha,r_0)r_0 \leq \alpha-|b|$ for some $b\in (-\alpha,\alpha)$,
the operator $\tilde{\mathcal{T}}\circ \mathcal{F}$ is an isomorphism from $\mathcal{H}_{b}(\mathbb{R})$ onto itself.
\end{enumerate}
\end{lemma}

\begin{proof}
For each $h\in \mathcal{H}_{b}(\mathbb{R})$,
$\mathcal{F}h$ is continuously differentiable by \eqref{df:F}.
Furthermore, a direct computation yields
\begin{eqnarray}\label{eq:Fh-dvt}
\mathcal{T}\circ \mathcal{F}h(t)=
\frac{d}{dt} \mathcal{F}h(t)-A(t)\mathcal{F}h(t)=h(t),
\end{eqnarray}
which implies assertion {\bf (i)}.
By \eqref{df:pert} and \eqref{eq:Fh-dvt},
\[
(\Lambda \circ \mathcal{F} h)(t)
=A(t)\int_{t-r}^{t}\left(A(\theta)\mathcal{F}h(\theta)+h(\theta)\right)\,d\theta
 +\int_{t-r}^{t}\eta(t,\theta-t)\left(A(\theta)\mathcal{F}h(\theta)+h(\theta)\right)\,d\theta.
\]
To estimate $\|\Lambda \circ \mathcal{F}h\|_{\mathcal{H}_{b}(\mathbb{R})}$,
by Lemma \ref{lm:F-adm}, \eqref{hyp:A} and the fact $h\in \mathcal{H}_{b}(\mathbb{R})$, we get
\begin{eqnarray*}
\int_{t-r}^{t}\Big| A(\theta)\mathcal{F}h(\theta)+h(\theta)\Big|\,d\theta
\!\!\! &\leq&\!\!\!
\int_{t-r_0}^{t} \|A(\theta)\|\|\mathcal{F}h\|_{\mathcal{Y}_{b}(\mathbb{R})}e^{-b|\theta|}\, d\theta
     + \int_{t-r_0}^{t} |h(\theta)| \,d\theta \\
\!\!\! &\leq&\!\!\!
r_0 e^{|b|r_0}\left(M \|\mathcal{F}h\|_{\mathcal{Y}_{b}(\mathbb{R})}+\|h\|_{\mathcal{H}_{b}(\mathbb{R})}\right)e^{-b|t|}.
\end{eqnarray*}
This together with \eqref{hyp-int} and $\eta(t,\theta) =-A(t)$ for  $\theta\leq -r$ yields that
\begin{eqnarray*}
&& \frac{1}{r_0}e^{b |t|}\int_{t}^{t+r_0}|(\Lambda \circ \mathcal{F} h)(u)|\,d u \\
&& \ \
\leq \frac{1}{r_0}e^{b |t|}\int_{t}^{t+r_0}\int_{u-r}^{u}\Big|\Big(A(u)+\eta(u,\theta-u)\Big)\Big( A(\theta)\mathcal{F}h(\theta)+h(\theta)\Big)\Big| \,d\theta \,d u \\
&& \ \
\leq e^{|b|r_0}\left(M \|\mathcal{F}h\|_{\mathcal{Y}_{b}(\mathbb{R})}+\|h\|_{\mathcal{H}_{b}(\mathbb{R})}\right)
     \int_{t}^{t+r_0}\|L(u)\| e^{b(|t|-|u|)}\, du  \\
&& \ \
\leq r_0 e^{2|b|r_0}M\left(M \|\mathcal{F}h\|_{\mathcal{Y}_{b}(\mathbb{R})}+\|h\|_{\mathcal{H}_{b}(\mathbb{R})}\right).
\end{eqnarray*}
Hence assertion {\bf (ii)} of this proposition is proved by Lemma \ref{lm:F-adm}.

Suppose that $r_0>0$ and $b\in (-\alpha,\alpha)$ satisfy $\ell(\alpha,r_0)r_0 \leq \alpha-|b|$.
A direct computation yields
\[
\frac{(\alpha-|b|)r_0}{1-e^{-(\alpha-|b|)r_0}} \leq 1.
\]
Subsequently, we have the following
\begin{eqnarray*}
\begin{aligned}
\frac{4KM^2 r_0 e^{4\alpha r_0}}{1-e^{-(\alpha-|b|)r_0}}+Me^{2\alpha r_0}
&= \frac{4KM^2 e^{4\alpha r_0}}{\alpha-|b|}\cdot \frac{(\alpha-|b|)r_0}{1-e^{-(\alpha-|b|)r_0}} +Me^{2\alpha r_0} \\
&\leq \frac{4KM^2 e^{4\alpha r_0}}{\alpha-|b|}+Me^{2\alpha r_0} \\
& < \frac{\ell(\alpha,r_0)}{\alpha-|b|}.
\end{aligned}
\end{eqnarray*}
Together with the assumption in {\bf (iii)}, we get
\[
r_0 \left\{\frac{4KM^2 r_0 e^{4\alpha r_0}}{1-e^{-(\alpha-|b|)r_0}}+Me^{2\alpha r_0}\right\}
<\frac{\ell(\alpha,r_0)r_0}{\alpha-|b|}\leq 1.
\]
By assertion {\bf (ii)} of this proposition,
\[
\|\Lambda \circ \mathcal{F}h\|_{\mathcal{H}_{b}(\mathbb{R})} <\frac{\ell(\alpha,r_0)r_0}{\alpha-|b|} \|h\|_{\mathcal{H}_{b}(\mathbb{R})}
\leq \|h\|_{\mathcal{H}_{b}(\mathbb{R})},
\ \ \ \forall\, r\in (0,r_0],\ \forall\, h\in \mathcal{H}_{b}(\mathbb{R}).
\]
This together with assertion {\bf (i)} yields that
\[
\mathcal{T}\circ \mathcal{F}+\Lambda \circ \mathcal{F}=Id+\Lambda \circ \mathcal{F},
\]
which is an isomorphism from $\mathcal{H}_{b}(\mathbb{R})$ onto itself.
This completes the proof.
\end{proof}

Finally, we emphasize that for each $r\in (0,r_0]$
the operator $\tilde{\mathcal{T}}\circ \mathcal{F}:\; \mathcal{H}_{b}(\mathbb{R})\to \mathcal{H}_{b}(\mathbb{R})$
can be regarded as $\mathcal{T}\circ \mathcal{F}$
under a small perturbation $\Lambda \circ \mathcal{F}$
if $r_0$ is sufficiently small.
This property will enable us to establish the proper admissibility of equation \eqref{DDE} with small delay.

\subsection{Proofs of Theorem \ref{thm-ED} and Theorem \ref{thm-ED2}}\

In the following,
we prove the main results on the robustness of exponential dichotomies against small-delay perturbations.
Here we assume that delay equation \eqref{DDE} satisfies  {\bf (A1)}\,-\,{\bf (A4)}.
Rewrite delay equation \eqref{DDE} as a small-delay perturbation of  ODE \eqref{eq:ODE}, i.e.,
\begin{eqnarray*}
\dot x =A(t)x(t)-\Big(A(t)x(t)-\int^{0}_{-r} d_\theta\eta(t,\theta)x(t+\theta) \Big)=A(t)x(t)-(\Lambda x)(t),
\end{eqnarray*}
where $A(t)$ is the coefficient matrix of \eqref{eq:ODE} and $\Lambda $ is defined below \eqref{df:T-pert}.
We will show that the small-delay perturbation $\Lambda x$ preserves exponential dichotomies,
and moreover, the dichotomy exponents undergo only slight changes under this perturbation.
The argument relies on our new criterion for exponential dichotomies of nonautonomous delay equations,
which is formulated through the proper admissibility of two pairs $(\mathcal{H}_{b^\pm}(\mathbb{R}),\mathcal{Y}_{b^\pm}(\mathbb{R}))$ with $b^-<0<b^+$,
and the operator perturbation analysis in the preceding discussion.

We are now ready to prove Theorem \ref{thm-ED}.

\begin{proof}[\itshape\bfseries Proof of Theorem \ref{thm-ED}]
By continuity, there exists $r_0>0$ such that $\ell(\alpha,r_0) r_0<\alpha$. Then  the interval $(0, \alpha-\ell(\alpha,r_0) r_0]$ is not empty.
Fix any $\beta \in (0, \alpha-\ell(\alpha,r_0) r_0]$. We can verify the following
\begin{eqnarray}\label{df:beta}
\beta>0, \ \ \
\pm \beta \in (-\alpha,\alpha), \ \ \
 \ell(\alpha,r_0) r_0 \leq \alpha-\beta.
\end{eqnarray}
We now prove the proper admissibility of the pair $(\mathcal{H}_{\beta}(\mathbb{R}),\mathcal{Y}_{\beta}(\mathbb{R}))$ with respect to \eqref{DDEnonhomo}.

Fix any $h \in \mathcal{H}_{\beta}(\mathbb{R})$.
By \eqref{df:beta} and {\bf (iii)} of Lemma \ref{lm:fdmentlm},
the linear operator $\tilde{\mathcal{T}}\circ \mathcal{F}$ defined in \eqref{df:TrF}  is an isomorphism
from $\mathcal{H}_{\beta}(\mathbb{R})$ onto itself for each $r\in (0,r_0]$.
Then there exists  $\tilde{h} \in \mathcal{H}_{\beta}(\mathbb{R})$ such that
\begin{eqnarray}\label{df:htild}
\tilde{\mathcal{T}}\circ \mathcal{F}\tilde{h}=h.
\end{eqnarray}
Let $y:=\mathcal{F}\tilde{h}$.
Then $y \in \mathcal{Y}_{\beta}(\mathbb{R})$ because of Lemma \ref{lm:F-adm}.
By \eqref{df:htild}, this function $y$ satisfies
\[
\dot y(t)-L(t)y_{t}=(\tilde{\mathcal{T}} y)(t)=h(t),
\]
implying that $(\mathcal{H}_{\beta}(\mathbb{R}),\mathcal{Y}_{\beta}(\mathbb{R}))$ is admissible.

We next prove that this $y$ is the unique solution of \eqref{DDEnonhomo} (or $\tilde{\mathcal{T}} y=h$).
It suffices to prove that if $x\in \mathcal{Y}_{\beta}(\mathbb{R})$ satisfies $\tilde{\mathcal{T}} x=0$ then $x=0$.
Since $x \in \mathcal{D}(\tilde{\mathcal{T}})$,
Lemma \ref{lm:EDAdmis-ODE} yields that there exists $\hat{h}\in \mathcal{H}_{\beta}(\mathbb{R})$ such that $x=\mathcal{F}\hat{h}$.
Then
$
\tilde{\mathcal{T}} (\mathcal{F}\hat{h})=\tilde{\mathcal{T}} x=0.
$
This together with {\bf (iii)} of Lemma \ref{lm:fdmentlm} yields $\hat{h}=0$.
Then $x=\mathcal{F}0=0$.
This proves the proper admissibility of $(\mathcal{H}_{\beta}(\mathbb{R}),\mathcal{Y}_{\beta}(\mathbb{R}))$.

An analogous argument also proves that $(\mathcal{H}_{-\beta}(\mathbb{R}),\mathcal{Y}_{-\beta}(\mathbb{R}))$ is properly admissible.
Therefore, the proof is completed by Proposition \ref{prop-adm-to-NED}.
\end{proof}

Finally, we prove Theorem \ref{thm-ED2} in the following.

\begin{proof}[\itshape\bfseries Proof of Theorem \ref{thm-ED2}]
Since $\zeta(\alpha)=Me^{4\alpha }(4KM+\alpha)$ and $\varepsilon_0=\min\{1,\, \alpha /\zeta(\alpha)\}$,
we have that
\[
\alpha-\zeta(\alpha) r> \alpha - \zeta(\alpha) \varepsilon_0 \geq \alpha - \zeta(\alpha) \cdot \frac{\alpha}{\zeta(\alpha)}=0
\]
for any $r\in (0,\varepsilon_0)$. This implies that the interval $(0, \alpha-\zeta(\alpha) r]$ is not empty.
Fix any $r\in (0,\varepsilon_0)$ and $\beta \in (0, \alpha-\zeta(\alpha) r]$.
Then $\beta>0$ and $\pm \beta \in (-\alpha,\alpha)$.
As in
the proof of Theorem \ref{thm-ED},
it suffices to prove that $\tilde{\mathcal{T}}\circ \mathcal{F}$ is an isomorphism from $\mathcal{H}_{\pm\beta}(\mathbb{R})$ onto
$\mathcal{H}_{\pm\beta}(\mathbb{R})$.
Noting that $\ell(\alpha,r_0)$, defined in \eqref{df:Kar0}, satisfies
\[
\ell(\alpha,\varepsilon_0)
=Me^{4\alpha \varepsilon_0}(4KM+\alpha)
\leq Me^{4\alpha }(4KM+\alpha)=\zeta(\alpha),
\]
we have
\begin{eqnarray}\label{est:sd-1}
\ell(\alpha,r) r<\ell(\alpha,\varepsilon_0) r \leq \zeta(\alpha) r
\leq \alpha-\beta.
\end{eqnarray}
Suppose that $\sup_{t\in\mathbb{R}}\|L(t)\|\leq M$. Then
\[
\sup_{t\in\mathbb{R}}\frac{1}{r} \int_{t}^{t+r}\|A(u)\| du
\leq
\sup_{t\in\mathbb{R}}\frac{1}{r} \int_{t}^{t+r}\|L(u)\| du
  \leq M<+\infty,
\]
implying that \eqref{hyp-int} holds for $r_0=r$.
Replacing $r_0$ in {\bf (iii)} of Lemma \ref{lm:fdmentlm} by $r$ and then using \eqref{est:sd-1},
we get that $\tilde{\mathcal{T}}\circ \mathcal{F}$  is an isomorphism from $\mathcal{H}_{\pm\beta}(\mathbb{R})$ onto $\mathcal{H}_{\pm\beta}(\mathbb{R})$.
This completes the proof.
\end{proof}

Remark that the robustness of exponential dichotomies against small-delay perturbations
for nonautonomous delay equations is well established in our paper.
This actually shows that small-delay perturbations cannot destroy the hyperbolic structure.
The proof relies on our novel admissible characterization.
Unlike previous existence results (e.g., Barreira-Valls \cite{BV-20} and Elorreaga-G\'{o}mez \cite{Elorreaga-Gomez-25}),
this characterization provides explicit estimates for the dichotomy exponents which is crucial for analyzing the persistence of the spectral gap under small-delay perturbations.
We believe that our admissible characterization is also applicable to the robustness problems associated with other types of perturbations.
The flexibility of our framework further suggests potential extensions to nonuniform or tempered exponential dichotomies,
provided that suitable weighted admissible spaces are introduced.






\bibliographystyle{amsplain}

\end{document}